\documentclass[final,1p,times]{elsarticle}
\usepackage{amssymb}
\usepackage{amsmath}
\usepackage{amsthm}
\usepackage[utf8]{inputenc}
\usepackage{marginnote}
\usepackage{amsfonts}
\usepackage{textcomp}
\usepackage{color,soul}
\usepackage{placeins}
\usepackage{tikz}
\usepackage{makecell}
\usepackage{amssymb}
\usepackage{graphicx}
\usepackage{hyperref}
\usepackage[none]{hyphenat}
\usepackage{lipsum}
\newtheorem{thm}{Theorem}[section]
\newtheorem{lem}[thm]{Lemma}
\newtheorem{rem}[thm]{Remark}
\newtheorem{prop}[thm]{Proposition}
\newtheorem{cor}[thm]{Corollary}
\newtheorem{example}[thm]{Example}
\newdefinition{defn}[thm]{Definition}
\allowdisplaybreaks

\journal{}
\begin{document}
	\begin{frontmatter}
		
		\title{Optimal $K-$dual frames and pairs in the presence of erasures}
		\author{Shankhadeep Mondal}
		\ead{shankhadeep.mondal@ucf.edu}
		\author{Deguang Han}
		\ead{deguang.han@ucf.edu}
		\author{R. N. Mohapatra}
		\ead{ram.mohapatra@ucf.edu}
		
		\address{Department of Mathematics, University of Central Florida}
		\cortext[shan]{Corresponding Author: shankhadeep.mondal@ucf.edu}		
		
		\begin{abstract}
		This paper explores the structure of optimal $K-$dual frames for a given $K-$frame and optimal $K-$dual pairs, within the context of erasures which occur during the transmission of frame coefficients. We address two distinct erasure scenarios and examine their impact on the reconstruction process. The optimality criteria are defined in terms of minimizing the spectral radius and the operator norm of the associated error operators. Through this approach, we provide a comprehensive framework for understanding and mitigating the effects of erasures in frame theory, contributing to enhanced robustness in data transmission and recovery.
		\end{abstract}

		
	\end{frontmatter}

\section{Introduction}
The concept of frames in Hilbert spaces were first introduced by Duffin and Schaeffer \cite{duff} in the context of addressing issues related to non-harmonic Fourier analysis. Frames can be understood as overcomplete systems that extend the concept of Riesz bases. The redundancy inherent in frames proves to be highly beneficial in various practical applications, including signal and image processing, sampling theory, data compression, where stability and flexibility in representation are crucial. In real-world communication system data loss often arises from unfavorable network conditions including channel noise, limited bandwidth or congestion. Such disruptions may cause the loss of certain frame coefficients during transmission, commonly referred to as \text{erasures.}. Frame redundancy plays a crucial role in counteracting these losses and improving the reliability of signal reconstruction. Motivated by these advantages, a substantial amount of recent research has been directed toward developing frame-based methods for enhancing data recovery in the presence of erasures.

Casazza and Kova{\v{c}}evi{'c} \cite{casa2} have thoroughly analyzed the robustness of equal-norm tight frames against erasures, highlighting their advantage over others. Similarly, Goyal, Kova{\v{c}}evi{'c}, and Kelner \cite{goya} explored uniform tight frames within a coding-theoretic framework, identifying them as optimal for single erasures. The notion of optimal dual frames which is crucial for minimizing maximum reconstruction errors, was introduced by Holmes and Paulsen \cite{holm}, who considered the operator norm as a measure of optimality. They established both necessary and sufficient conditions under which Parseval frames and their canonical dual pairs achieve minimal reconstruction error for any possible location of erasures. In \cite{bodm1}, the problem was studied by defining the reconstruction error numerically as a type of average—specifically, averaging the operator norms of the error operators over all possible erasure patterns of a fixed size and then minimizing this average. Subsequent research has further refined these concepts. Lopez and Han \cite{jerr} investigated sufficient conditions ensuring that the canonical dual frame is the uniquely optimal dual in the case of both one and more erasures. In the context of $1-$erasure, Pehlivan, Han and Mohapatra \cite{sali} identified dual frames that minimize the maximum error by taking spectral radius as error measurement of the associated error operators. This line of work was subsequently extended to address the case of two erasures in \cite{peh,dev}. Given the inherent randomness in erasure events, a probabilistic model for analyzing such problems is indispensable. Leng, Han, and Huang \cite{leng3, leng} incorporated probability-weighted models to explore optimality under operator norms, providing conditions for probabilistic optimality for erasures. Optimal dual frames under probabilistic erasures were further studied in \cite{shan4, li2}.

In recent years, numerous generalized frame structures have emerged, including fusion frames, group representation frames, and \(K-\)frames. Among these, \(K-\)frames introduced by G\u{a}vru\c{t}a \cite{gua}, have received considerable attention for their utility in theoretical analysis and practical applications, particularly in the context of data loss and signal recovery. A \(K-\)frame extends the classical notion of frames by allowing reconstruction only within the range of a bounded linear operator \(K\). Several authors have contributed to the structural and constructive development of \(K\)-frames and their duals. For example, \cite{du} presents new constructions of \(K\)-\(g\)-frames by exploiting the interplay between the frame operator and a positive operator, providing necessary and sufficient conditions for their existence. The paper \cite{he} studies \(K-\)frames generated via operator iterates of the form \(\{T^k f_0\}_{k \in \mathbb{Z}}\), focusing on the boundedness and stability of the operator \(T\) in relation to \(K\). In \cite{li4}, the authors explore strategies for constructing \(K\)-dual frame pairs and identifying common \(K-\)dual Bessel sequences. Generalized approaches appear in \cite{bar}, where dual systems are studied in the setting of Hilbert \(C^*\)modules, including conditions under which the sum of two such frames forms a valid \(c\)-\(K\)-\(g\)-frame. The work in \cite{wang} investigates additive properties of \(K-\)frames in \(n\)-Hilbert spaces, while \cite{nami} introduces continuous \(K\)-frames (or \(ck-\)frames) and provides methods for constructing their duals. Additionally, \cite{li3} offers a characterization of \(K\)-\(g\)-frames using quotient operator theory, deriving necessary and sufficient conditions for duality. Together, these contributions advance the landscape of \(K-\)frame theory and lay the groundwork for further exploration into optimal dual design under erasure and reconstruction constraints. The literature on $K-$frames continues to grow, with additional insights and generalizations found in the works of various authors, such as those in \cite{pill,xiao,koba}.

This paper extends these studies by focusing on two key aspects of the erasure problem:
\begin{itemize}
	\item \textbf{Optimal $K-$dual pairs:} We investigate the existence and characterization of $K-$dual pairs that minimize the maximum reconstruction error across all possible erasure patterns, with error measured by the operator norm and the spectral radius.
	
	\item \textbf{Optimal $K-$dual frames:} For a fixed $K-$frame, we examine the existence and structure of optimal $K-$dual frames under the same error measurement.
\end{itemize}

Section~2 develops the foundational framework for modeling coefficient erasures in $K-$frame systems. Section~3 addresses the pairwise optimization problem, where we analyze all possible $(N,n)$ $K-$dual pairs $(F, G)$. We introduce the notions of \textit{$1-$uniform} and \textit{$2-$uniform} $K-$dual pairs, and show that $1-$uniformity ensures optimality for one erasure with respect to both the operator norm and the spectral radius. On the other hand, $2-$uniformity guarantees optimality for two erasures with respect to the spectral radius. We derive tight lower bounds for the maximum error in the two-erasure case and identify conditions under which these bounds are attained. In particular, we show that if a $2-$uniform $K-$dual pair exists and it achieves the optimality, which we express explicitly in terms of the trace and Hilbert–Schmidt norm of the operator $K$. We also prove that both operator norm and spectral radius optimality are preserved under unitary transformations that commute with $K$. Section~4 is devoted the fixed-frame optimization problem, where we fix a given Parseval $K-$frame $F$ and consider all possible $K-$dual frames $G$ of $F$. We define the maximum reconstruction error for one erasure using both the operator norm and the spectral radius, and derive necessary and sufficient conditions for the canonical $K-$dual of a Parseval $K-$ frame, namely $K^\dag F,$ to be optimal.  Furthermore, we show that when the number of frame elements exceeds the dimension of the space ($N > n$), there exist uncountably many $K$-duals of $F$ that achieve the same optimal error, illustrating the non-uniqueness of optimal duals in this setting. We also provide carefully constructed examples of Parseval $K-$frames and their duals to support and clarify the theoretical results. These examples highlight both cases where the canonical dual is the unique optimal choice and cases where multiple optimal duals exist. Our findings not only extend classical results in frame theory to the setting of \( K-\)frames, but also provide a rigorous framework for analyzing optimal reconstruction in the presence of erasures, which is an area of fundamental importance in signal processing, coding theory, and various other engineering disciplines.

	\section{Preliminaries  }

% Preliminaries

Let \( \mathcal{H}_n \) be an \( n \)-dimensional Hilbert space and let \( \mathcal{B}(\mathcal{H}_n) \) denote the Banach algebra of all bounded linear operators on \( \mathcal{H}_n \). A finite sequence of vectors \( F = \{f_i\}_{i=1}^N \subset \mathcal{H}_n \), with \( N \geq n \), is said to form a \textit{frame} for \( \mathcal{H}_n \) if there exist positive constants \( A \) and \( B \) such that
$$ A \|f\|^2 \leq \sum_{i=1}^N |\langle f, f_i \rangle|^2 \leq B \|f\|^2, \quad \text{for all } f \in \mathcal{H}_n. $$
The values \( A \) and \( B \) are referred to as the lower and upper frame bounds respectively. A frame is called \textit{tight} if \( A = B \) and \textit{Parseval} if \( A = B = 1 \). Given a frame \( F = \{f_i\}_{i=1}^N \) for the Hilbert space \( \mathcal{H}_n \), the associated analysis operator is the map \( \Theta_F : \mathcal{H}_n \to \mathbb{C}^N \), defined by \( \Theta_F(f) = \{\langle f, f_i \rangle\}_{i=1}^N \). Its adjoint \( \Theta_F^* : \mathbb{C}^N \to \mathcal{H}_n \) is called the synthesis operator and satisfies \( \Theta_F^*(\{c_i\}) = \sum_{i=1}^N c_i f_i \). The frame operator \( S_F = \Theta_F^* \Theta_F \) is given by \( S_F f = \sum_{i=1}^N \langle f, f_i \rangle f_i \), is self-adjoint, positive and invertible. A sequence \( G = \{g_i\}_{i=1}^N \subset \mathcal{H}_n \) is a dual frame of \( F \) if every \( f \in \mathcal{H}_n \) can be reconstructed as \( f = \sum_{i=1}^N \langle f, f_i \rangle g_i \), equivalently \( \Theta_G^* \Theta_F = I \). The canonical dual of \( F \) is \( \{S_F^{-1} f_i\} \) and all duals are of the form \( \{S_F^{-1}f_i + u_i\} \), where \( \{u_i\} \) satisfies a biorthogonality condition i.e. $\sum_{i=1}^N \langle f, f_i \rangle u_i = \sum_{i=1}^N \langle f, u_i \rangle f_i =0,$ for all $f \in \mathcal{H}_n.$ Moreover, any dual pair \( (F, G) \) satisfies the relation  \( \sum_{i=1}^N \langle f_i, g_i \rangle = \mathrm{tr}(I) = n \). For further details, we refer to \cite{ole}.

Now, for any bounded operator \( K \in \mathcal{B}(\mathcal{H}_n) \), a sequence \( F = \{f_i\}_{i=1}^N \subset \mathcal{H}_n \) is called a \( K-\)\textit{frame} if there exist constants \( 0 < A \leq B < \infty \) such that
\[
A\|K^*f\|^2 \leq \sum_{i=1}^N |\langle f, f_i \rangle|^2 \leq B\|f\|^2, \quad \forall f \in \mathcal{H}_n.
\]
If equality holds in the lower bound, i.e., \( \sum_{i=1}^N |\langle f, f_i \rangle|^2 = A\|K^*f\|^2 \), then \( F \) is called a \textit{tight} \( K-\)frame. In particular, when \( A = 1 \), the frame is a \textit{Parseval} \( K -\)frame.

For a \( K-\)frame \( F = \{f_i\}_{i=1}^N \), we define the associated operators as follows:
\begin{itemize}
  \item The \textit{analysis operator} \( \Theta_F^K : \mathcal{H}_n \to \mathbb{C}^N \) is defined by
  \[ \Theta_F^K(f) = \{\langle f, f_i \rangle\}_{i=1}^N. \]

  \item The \textit{synthesis operator} \( (\Theta_F^K)^* : \mathbb{C}^N \to \mathcal{H}_n \) is given by
  \[ (\Theta_F^K)^*(\{c_i\}_{i=1}^N) = \sum_{i=1}^N c_i f_i. \]

  \item The \textit{frame operator} \( S_F^K : \mathcal{H}_n \to \mathcal{H}_n \) is defined as the composition
  \[ S_F^K f = (\Theta_F^K)^* \Theta_F^K f = \sum_{i=1}^N \langle f, f_i \rangle f_i. \]
\end{itemize}

If \( K \) has closed range and \( F = \{f_i\}_{i=1}^N \) is a Parseval \( K-\)frame, then the sequence \( \{K^\dag f_i\}_{i=1}^N \) forms a \textit{canonical} \( K-\)dual of \( F \), where \( K^\dag \) is the Moore--Penrose pseudoinverse of \( K \). A Bessel sequence \( G = \{g_i\}_{i=1}^N \) is a \( K-\)dual of the \( K-\)frame \( F \) if it satisfies
\begin{equation}\label{eq:k-dual}
Kf = \sum_{i=1}^N \langle f, g_i \rangle f_i, \quad \forall f \in \mathcal{H}_n.
\end{equation}
From \eqref{eq:k-dual}, it follows that \( K^*f = \sum_{i=1}^N \langle f, f_i \rangle g_i. \) In general, the sequences \( \{f_i\} \) and \( \{g_i\} \) are not interchangeable in \eqref{eq:k-dual} unless \( K \) is self-adjoint. A pair \( (F, G) \) of \( N \) element sequences is called an \( (N, n) \) \( K -\)dual pair if \( G \) is a \( K-\)dual of the $K-$frame \( F \) and \( F \) is a \( K^* -\)dual of \( G \).
\begin{rem}
If \( G \) is a \( K-\)dual of \( F \), then \( F \) is a \( K^*-\)dual of \( G \).
\end{rem}
\noindent In practical scenarios, due to data loss some partial coefficients of \( \Theta_F^K(f) \) may be unavailable. Let \( E \) be a diagonal matrix with entries 0 or 1 indicating erased or received coefficients, and define \( D = I - E \). The reconstructed signal is given by
\[
\widehat{Kf} = (\Theta_F^K)^* E \Theta_G^K f,
\]
and the corresponding reconstruction error becomes
\[
Kf - \widehat{Kf} = (\Theta_F^K)^* D \Theta_G^K f.
\]
If \( \Lambda \) denotes the index set of the erased coefficients, then
\[ Kf - \widehat{Kf} = \sum_{i \in \Lambda} \langle f, g_i \rangle f_i. \]

For a \( K-\)frame \( F \) and a \( K -\)dual \( G \), the worst-case error for all possible \( m-\)erasures is expressed as
\[ \max \left\{ \mathcal{M}\left((\Theta_G^K)^* D \Theta_F^K\right) : D \in \mathcal{D}_m \right\}, \]
where \( \mathcal{M} \) is an appropriate error measure and \( \mathcal{D}_m \) is the collection of all diagonal matrices with exactly \( m \) ones.

    \section{Characterization of Optimal $K-$Dual Pairs}
    
Our objective in this section is to determine the optimal \( K-\)dual pairs that minimize the maximum reconstruction error. To facilitate this analysis, we introduce the concepts of \(1-\)uniform and \(2-\)uniform \(K-\)duals, which play a crucial role in studying optimality with respect to both the operator norm and the spectral radius of the corresponding error operator.

\begin{defn}\label{defn3point1}
Let \( F = \{f_i\}_{i=1}^N \) be a \( K-\)frame for \( \mathcal{H}_n \). A \( K-\)dual \( G = \{g_i\}_{i=1}^N \) of \( F \) is called a \emph{$1-$uniform \( K-\)dual} if there exists a constant \( c \in \mathbb{C} \) such that $\langle f_i, g_i \rangle = c \quad \text{for all } 1 \leq i \leq N.$ In this case, the pair \( (F, G) \) is called \emph{$1-$uniform \( K-\)dual pair}.
\end{defn}

\begin{rem}\label{rem3point2}
If $(F,G)$ is a $1-$uniform $K-$dual pair, then the constant  \( c  \) turns out to be \(\dfrac{\operatorname{tr}(K)}{N}\), since
\[
\operatorname{tr}(K) = \operatorname{tr}\left( (\Theta_F^K)^* \Theta_G^K \right) = \operatorname{tr}\left( \Theta_G^K (\Theta_F^K)^* \right) = \sum_{i=1}^N \langle f_i, g_i \rangle = cN.
\]
\end{rem}

\begin{defn}
 A $1-$uniform \( K-\)dual pair \( (F, G) \) is called \emph{$2-$uniform \( K -\)dual pair} if there exists a constant \( c' \in \mathbb{C} \) such that $\langle f_i, g_j \rangle \langle f_j, g_i \rangle = c', \quad \text{for all } 1 \leq i, j \leq N \text{ with } i \neq j.$ In this case, the sequence \( G \) is called a \emph{$2-$uniform \( K -\)dual} of \( F \).
\end{defn}

\subsection{Operator norm based optimality of $K-$dual pairs}

In this subsection, we study optimality of an $(N,n)$ $K-$dual pair $(F, G)$ for the Hilbert space $\mathcal{H}_n$ by using the operator norm as the error measurement \( \mathcal{M} \). Specifically, we define
\begin{align*}
	&\mathbb{O}_{1}(F,G) := \max \left\{ \left\| (\Theta_{G}^K)^* D \Theta_F^K \right\| : D \in \mathcal{D}_1 \right\}, \\&
	\widetilde{\mathbb{O}}_{1} := \inf \left\{ \mathbb{O}_{1}(F,G) : (F,G)\; \text{is an } (N,n)\; K\text{-dual pair for } \mathcal{H}_n \right\}.
\end{align*}
Here, \( \mathbb{O}_{1}(F,G) \) quantifies the worst-case reconstruction error for $1-$erasures and \( \widetilde{\mathbb{O}}_{1} \) represents the infimum of this error across all possible $(N,n)$ $K$-dual pairs. A $K-$dual pair \( (F', G') \) is said to be a \textit{$1-$erasure optimal with respect to the operator norm} if $\mathbb{O}_{1}(F', G') = \widetilde{\mathbb{O}}_{1}.$
\par For a given $(N,n)$ $K$-dual pair $(F, G)$, if the erasure occurs in the $i-$th position, the associated operator norm of the error operator becomes $\left\| (\Theta_{G}^K)^* D \Theta_{F}^K \right\| = \sup\limits_{\|f\|=1} \| \langle f, f_i \rangle g_i \| = \|f_i\| \cdot \|g_i\|.$ Therefore, we have the identity:
\begin{equation}\label{equation4}
	\mathbb{O}_{1}(F, G) = \max \bigg\{ \|f_i\| \cdot \|g_i\| : 1 \leq i \leq N \bigg\}.
\end{equation}

We now proceed to characterize those $(N,n)$ $K-$dual pairs that achieve $1-$erasure optimality with respect to the operator norm.

\begin{prop}\label{thm4point1}
	Let \( K \in \mathcal{B}(\mathcal{H}_n) \) be a positive operator. Then, $\widetilde{\mathbb{O}}_{1} = \dfrac{\operatorname{tr}(K)}{N}.$ Moreover, an $(N,n)$ $K-$dual pair \( (F, G) \) is a $1-$erasure optimal $K-$dual pair with respect to the operator norm if and only if \,$\|f_i\| \cdot \|g_i\| = \dfrac{\operatorname{tr}(K)}{N}, \quad \text{for all } 1 \leq i \leq N.$
\end{prop}

\begin{proof}
	
	Suppose  $\mathbb{O}_{1}  (F,G) < 1,$\; for some $(N,n)$ $K-$dual pair $(F,G).$ Then, $ \|f_i\|\,\|g_i\| < \frac{tr(K)}{N} ,$ for all $1 \leq i \leq N.$ This leads to\, $tr(K)= \left|\sum\limits_{i=1}^N \langle f_i , g_i \rangle \right| \leq \sum\limits_{i=1}^N   \|f_i\|\,\|g_i\| < \sum\limits_{i=1}^N \dfrac{tr(K)}{N} = tr(K),$ \; which is not possible. Thus, $\mathbb{O}_{1}  (F,G) \geq\dfrac{tr(K)}{N},$ for every $K-$dual pair $(F,G)$.  Let $F'$ be a uniform Parseval frame, whose existence is guaranteed by Theorem 2.1 in \cite{cass2}. In particular, we consider the case where \( S \) is the identity operator, whose spectrum consists only of the eigenvalue \( 1 \) and we take \( a_i = \sqrt{\frac{n}{N}} \) for all \( i \).
 As $\sum\limits_{i=1}^k a_{i}^2 = \frac{kn}{N} \leq k,$ for $ 1\leq k \leq n$ and  $\sum\limits_{i=1}^N a_{i}^2 = n,$ there exists a Parseval frame $\{f'_i\}_{i=1}^N$ such that $\| f'_i\|^2 = \frac{n}{N}.$ Then $K^{\frac{1}{2}}F'$ is a $K-$frame with $K-$dual $K^{\frac{1}{2}}F',$ since $\sum\limits_{i=1}^N \langle f, K^{\frac{1}{2}}f'_i \rangle K^{\frac{1}{2}}f'_i = K^{\frac{1}{2}} \sum\limits_{i=1}^N \langle K^{\frac{1}{2}}f, f'_i \rangle f'_i =Kf,$ for all $f \in \mathcal{H}_n.$ Also, $\|K^{\frac{1}{2}} f'_i\| \leq \|K^{\frac{1}{2}}\|\,\|f'_i\|$ and $\|f'_i\| = \|K^{-\frac{1}{2}} K^{\frac{1}{2}} f'_i \| \leq \|K^{-\frac{1}{2}}\|\,\|K^{\frac{1}{2}} f'_i\|,$ for all $i.$ Thus, $\|K^{\frac{1}{2}} f'_i\| = \|K^{\frac{1}{2}}\|\,\|f'_i\|$ and hence, $K^{\frac{1}{2}}F'$ is a uniform  $K-$frame for $\mathcal{H}_n$ with $K-$dual $K^{\frac{1}{2}}F'.$ Using the fact that  $\sum\limits_{i=1}^N \langle K^{\frac{1}{2}} f'_i , K^{\frac{1}{2}} f'_i \rangle = tr (K)$ we have $\| K^{\frac{1}{2}}f'_i\| = \sqrt{\dfrac{tr(K)}{N}}.$ Therefore, $\mathbb{O}_{1}  (K^{\frac{1}{2}}F',K^{\frac{1}{2}}F') = \max\limits_{1 \leq i \leq N} \left\| K^{\frac{1}{2}}f'_i \right\|^2  = \dfrac{tr(K)}{N}.$
    
    \par  Now, let $(F,G) $ be a $1-$erasure optimal $K-$dual pair.  Then, $\max\limits_{1 \leq i \leq N}  \|f_i\|\;\|g_i \| = \frac{tr(K)}{N}.$ If for any $j \in \{1,2,\ldots,N\},$ $ \|f_j\|\;\|g_j  \| < \frac{tr(K)}{N},$  then $tr(K) = \left|\sum\limits_{i=1}^N \langle f_i, g_i \rangle \right|  \leq \sum\limits_{i=1}^N \|f_i\|\;\|g_i \| < \sum\limits_{i=1}^N \dfrac{tr(K)}{N} =tr(K) ,$ which is a contradiction.	Therefore, $ \|f_i\|\;\|g_i \| = \frac{tr(K)}{N},\, \forall\; i.$ Conversely, if   $\|f_i\|\;\|g_i\| = \dfrac{tr(K)}{N},\;\forall \,1 \leq i \leq N,$ then obviously, $\mathbb{O}_{1} (F,G) = \dfrac{tr(K)}{N}$ and hence $(F,G) $ is a $1-$erasure optimal $K-$dual pair.
	
\end{proof}

	\begin{rem}\label{Thm3point5}
For any positive operator $K \in \mathcal{B}(\mathcal{H}_n),$ there always exists a $K$-frame $F$ such that the $K-$dual pair $(F, F)$ is a $1-$erasure optimal $K-$dual pair under the operator norm.
\end{rem}

The following corollary establishes a link between $1-$erasure optimality under the operator norm and the $1-$uniformity of a $K-$dual pair.

\begin{cor}\label{cor3point5}
Let \( K \in \mathcal{B}(\mathcal{H}_n) \) be a positive operator. Then any \( K-\)dual pair that is $1-$erasure optimal with respect to the operator norm is necessarily  a $1-$uniform \( K-\)dual pair.
\end{cor}

\begin{proof}
	Let \((F, G)\) be a $1-$erasure optimal \(K-\)dual pair. It follows from Proposition~\ref{thm4point1} that, $\|f_i\| \, \|g_i\| = \dfrac{\operatorname{tr}(K)}{N}, \quad \forall i.$ So,
	$$ tr(K) = \displaystyle{\left|\sum_{i=1}^N \langle f_i , g_i \rangle \right|} \leq   \displaystyle{\sum_{i=1}^N \left| \langle f_i , g_i \rangle \right|} \leq \displaystyle{\sum_{i=1}^N \|f_i \|\; \|g_i \| } = \displaystyle{\sum_{i=1}^N \frac{tr(K)}{N} = tr(K) }.$$
	This implies that $\left| \langle f_j , g_j \rangle \right|  =   \|f_j \| \,\|g_j \|, \;\forall\, 1 \leq j\leq N.$	Taking $\langle f_j , g_j \rangle = a_j + ib_j ,\,\text{where} \;a_j,b_j \in \mathbb{R},\;  1 \leq j\leq N,$ we get  $\displaystyle{\sum_{j=1}^N a_j = tr(K) }$,  $\displaystyle{\sum_{j=1}^N b_j =0 }$ \;\;and $\sqrt{a_j^2 + b_j^2} = \dfrac{tr(K)}{N}, \;\;  1 \leq j\leq N. $	This gives
	$$ tr(K)= \displaystyle{\sum_{j=1}^N a_j } \leq  \displaystyle{\sum_{j=1}^N \sqrt{a_j^2 + b_j^2} } =  \displaystyle{\sum_{j=1}^N \frac{tr(K)}{N}} = tr(K).$$
	This shows that $b_j =0$\,and\,$a_j = |a_j|, \, \forall \,j. $ Therefore, $ \langle f_j , g_j \rangle =\frac{tr(K)}{N},$ \;$\forall \,1 \leq j \leq N.$
	
\end{proof}

\noindent It is well known that \((F, G)\) is a \(K-\)dual pair if and only if \((UF, UG)\) is also a \(K-\)dual pair, where \(U\) is a unitary operator on \(\mathcal{H}_n\) that commutes with \(K\). The following remark demonstrates that the collection of all \(1-\)erasure optimal \(K-\)dual pairs with respect to the operator norm is invariant under the action of such unitary operators.

\begin{rem}\label{thm3point7}
Let \( U \) be a unitary operator on \( \mathcal{H}_n \) that commutes with \( K \), i.e., \( UK = KU \). Then a $K-$dual pair \( (F, G) \) is  $1-$erasure optimal with respect to the operator norm if and only if \( (UF, UG) \) is a $1-$erasure optimal $K-$dual pair with respect to the operator norm.
\end{rem}

\noindent	The proof of the above remark is straightforward, as ${\mathbb{O}}_{1} (F,G) = {\mathbb{O}}_{1} (UF,UG) .$

\vspace{1.5mm}

\subsection{Spectral Radius based optimality of $K-$Dual Pair}

Here we take the error measure \( \mathcal{M} \) to be the spectral radius \( \rho \) of the corresponding error operator. For an $(N,n)$ $K-$dual pair $(F,G)$ for $\mathcal{H}_n,$ let us define
:
\begin{align*}
	&r_{m} (F,G) :=\max \bigg\{\rho\left( (\Theta_{G}^K)^* D\Theta_{F}^K\right) : D \in \mathcal{D}_{m} \bigg\},\; 1 \leq m \leq N\\
	& \tilde{r}_{1}  := \inf \bigg\{ r_{1} (F,G) : \textit{$(F,G)$ \text{is an $(N,n)\;$  $K-$dual pair for} $\mathcal{H}_n$} \bigg\}\\
	&\mathcal{R}_{1} := \bigg\{ (F,G) : r_{1} (F,G) = \tilde{r}_{1} \bigg\}\\
	&\tilde{r}_{m} := \inf \bigg\{ r_{m} (F,G) : (F,G) \in \mathcal{R}_{m-1} \bigg\},\;\; 1< m \leq N\\
	& \mathcal{R}_{m} := \bigg\{(F,G)\in \mathcal{R}_{m-1} : r_{m} (F,G) =  \tilde{r}_{m} \bigg\}, \;\; 1 < m \leq N.
\end{align*}
The members of \( \mathcal{R}_{m} \), for \( 1 \leq m \leq N \), are referred to as \textit{$m-$erasure spectrally optimal \( K-\)dual pairs}. In the specific case when \( m = 1 \), it is straightforward to observe that if the erasure occurs at the \( i^{\text{th}} \) position, then the spectral radius of the error operator satisfies $\rho\left((\Theta_{G}^K)^* D \Theta_{F}^K\right) = |\langle f_i , g_i \rangle|,$ since the matrix \( (\Theta_{G}^K)^* D \Theta_{F}^K \) has eigenvalues \( 0 \) and \( \langle g_i, f_i \rangle \).
Therefore,
\begin{align}\label{equation3point1}
	r_{1} (F,G) = \max \bigg\{ |\langle f_i , g_i \rangle | : 1\leq i \leq N \bigg\} .
\end{align}
We now derive an useful and explicit expression for \( r_2(F, G) \), which quantifies the maxium error under two erasures. Consider
\begin{eqnarray*}
	r_{2} (F,G) = \max \left\{ \rho\left((\Theta_{G}^K)^* D\Theta_{F}^K \right) : D \in \mathcal{D}_{2} \right\} = \max \left\{ \rho \left(D\Theta_{F}^K (\Theta_{G}^K)^* \right) : D \in \mathcal{D}_{2} \right\}.
\end{eqnarray*}
By setting  $D = diag [0,0,\ldots,0,1,0,\ldots,0,1,0,\ldots,0],$ the matrix representation $\mathcal{A}$ of $D\Theta_{F}^K (\Theta_{G}^K)^*$ with respect to the standard orthonormal basis of $\mathbb{C}^N$ is

%$$A = \left[\begin{array}{cccccccc}
	$$	\mathcal{A} = \begin{pmatrix}
		0 & 0  & \cdots &0 & \cdots &0 & \cdots &0\\
		\vdots &\vdots  &\vdots &\vdots &\vdots&\vdots &\vdots &\vdots \\
		0 & 0  & \cdots &0 & \cdots &0 & \cdots &0 \\
		\alpha_{1i} & \alpha_{2i} & \cdots &\alpha_{ii} & \cdots &\alpha_{ji} & \cdots &\alpha_{Ni} \\
		0 & 0  & \cdots &0 & \cdots &0 & \cdots &0\\
		\vdots &\vdots  &\vdots &\vdots &\vdots&\vdots &\vdots &\vdots \\
		0 & 0  & \cdots &0 & \cdots &0 & \cdots &0 \\
		\alpha_{1j} & \alpha_{2j} & \cdots &\alpha_{ij} & \cdots &\alpha_{jj} & \cdots &\alpha_{Nj} \\
		0 & 0  & \cdots &0 & \cdots &0 & \cdots &0\\
		\vdots &\vdots  &\vdots &\vdots &\vdots&\vdots &\vdots &\vdots \\
		0 & 0  & \cdots &0 & \cdots &0 & \cdots &0 \\
	\end{pmatrix},$$
	%\end{array}\right].$$
	where $\alpha_{ij}:= \langle g_i , f_j \rangle, \forall\,i,j. $
	\noindent
	The characteristic polynomial of the above matrix $\mathcal{A}$ is given by
	\begin{align*}
		\det(\mathcal{A}-xI) &= (-x)^{i-1}(-x)^{N-j} \left|\begin{array}{ccccccc}
			-x +  \alpha_{ii} &q_i\alpha_{(i+1) i} &\cdots &\alpha_{(j-1)i} &\alpha_{ji}\\
			0 & -x  &\cdots& 0 &0\\
			\vdots &\vdots  &\vdots &\vdots &\vdots \\
			0 & 0  &\cdots& -x &0\\
			\alpha_{ij} & \alpha_{(i+1)j} & \cdots & \alpha_{(j-1)j} & -x + \alpha_{jj} \\
		\end{array}\right| \\ &= (-x)^{N-2} \left((x- \alpha_{ii})(x-\alpha_{jj}) -  \alpha_{ij} \alpha_{ji} \right).
	\end{align*}
	\noindent
	Therefore, the eigenvalues of $\mathcal{A}$ are $0,\, \dfrac{ \alpha_{ii} + \alpha_{jj} \pm \sqrt{( \alpha_{ii}-  \alpha_{jj})^2 + 4 \alpha_{ij} \alpha_{ji} }}{2}. $ Hence,
	\begin{equation} \label{eqn3point3}
		r_{2} (F,G) = \max_{i\neq j} \left| \frac{ \alpha_{ii} + \alpha_{jj} \pm \sqrt{( \alpha_{ii} -  \alpha_{jj})^2 + 4 \alpha_{ij}\alpha_{ji} }}{2} \right|.
	\end{equation}
	The next theorem provides a complete characterization of when a \( K-\)dual pair is $1-$erasure spectrally optimal.

	\begin{thm} \label{thm5point2}
   Let $K \in \mathcal{B}(\mathcal{H}_n)$ be a positive semi-definite operator. Then, the value of $\tilde{r}_{1} $ is $\frac{tr(K)}{N}.$  Moreover, an $(N,n)$ dual pair $(F,G) \in \mathcal{R}_{1} $ if and only if it is a $1-$uniform $K-$dual pair.
		
	\end{thm}
	\begin{proof}
		For any $K-$dual pair $(F,G),$ $\sum\limits_{i=1}^N \left|\langle f_i, g_i \rangle \right| \geq \sum\limits_{i=1}^N \langle f_i, g_i \rangle = tr(K) >0  $ and hence by \eqref{equation3point1}, $r_{1}(F,G) \geq \frac{tr(K)}{N}.$ From the proof of Proposition~\ref{thm4point1}, it follows that there exists a Parseval frame \( F = \{f_i\}_{i=1}^N \) for \( \mathcal{H}_n \) such that \( \|f_i\|^2 = \frac{n}{N} \) for all \( i \), and the sequence \( \{K^{1/2} f_i\}_{i=1}^N \) constitutes a \( K-\)frame whose \( K-\)dual is given by \( K^{1/2} F \). Therefore,
		$\langle K^{\frac{1}{2}}f_i ,K^{\frac{1}{2}}f_i \rangle = \left\| K^{\frac{1}{2}} f_i \right\|^2 = \left\| K^\frac{1}{2} \right\|^2 \,\left\| f_i \right\|^2  = \frac{tr(K)}{N}$ and hence, $r_{1}\left(K^{\frac{1}{2}}F,K^{\frac{1}{2}} F \right) = \frac{tr(K)}{N}.$ Thus, $\tilde{r}_{1} = \dfrac{tr(K)}{N}.$
		
		\par Now, suppose $(F',G') \in \mathcal{R}_{1}.$ Then, $\max\limits_{1 \leq j \leq N} \; |\langle f'_j , g'_j \rangle | = \dfrac{tr(K)}{N}.$ If for any $j \in \{1,2,\ldots,N\},\;$ \;$|\langle f'_j , g'_j \rangle | < \dfrac{tr(K)}{N},$ then  $ tr(K) = \displaystyle{\sum_{j=1}^N \langle f'_j , g'_j \rangle \leq \sum_{j=1}^N |\langle f'_j , g'_j \rangle | < \sum_{j=1}^N \frac{tr(K)}{N}} = tr(K)  ,$ which is not possible. Therefore, $|\langle f'_j, g'_j \rangle| =\dfrac{tr(K)}{N},\, \forall\, j.$ It is enough to verify that $\langle f'_j, g'_j \rangle \geq 0,\,\forall\, j.$ Clearly,
		$$\sum\limits_{j=1}^N \langle f'_j , g'_j \rangle =tr(K)= \sum\limits_{j=1}^N |\langle f'_j , g'_j \rangle | .$$
		Let $\langle f'_j , g'_j \rangle = a_j + ib_j,\, 1\leq j \leq N,$ where $a_j,b_j \in \mathbb{R}.$ Then, $\sum\limits_{j=1}^N a_j = tr(K) = \sum\limits_{j=1}^N \sqrt{a_j^2 +b_j^2,}$ which is possible when $b_j = 0$  and also $a_j \geq 0,\, \forall\, j.$ Conversely, if a $K-$dual pair $(F',G')$ is  $1-$uniform , then by Remark \ref{rem3point2}, $r_{1}(F',G') = \frac{tr(K)}{N},$ which implies that $(F',G') \in \mathcal{R}_{1}.$
		
	\end{proof}

	\begin{rem}\label{Thm3point5}
	There exists a $K$-frame \( F \) in \( \mathcal{H}_n \) such that the pair \( (F, F) \in \mathcal{R}_{1} \).
\end{rem}

\begin{rem}
	The preceding theorem guarantees the existence of an optimal $K-$dual pair that minimize both the spectral radius and the operator norm of the error operator, even when considering the entire class of $K-$dual pairs.
\end{rem}

	We now present a characterization of  spectrally optimal dual pairs for two erasures. Recall that the spectral radius for two erasures is given by
\[
r_{2}(F,G) =  \max\limits_{i\neq j} \left| \dfrac{ \alpha_{ii} + \alpha_{jj} \pm \sqrt{( \alpha_{ii} -  \alpha_{jj})^2 + 4  \alpha_{ij}\alpha_{ji} }}{2} \right|.
\]
Moreover, if \( (F,G) \in \mathcal{R}_{1} \), then by Theorem~\ref{thm5point2}, the expression for \( r_{2}(F,G) \) simplifies to
\begin{align}\label{eqn3point4*}
	r_{2}(F,G) = \max_{i\neq j} \left| \dfrac{\operatorname{tr}(K)}{N} \pm \sqrt{\alpha_{ij}\alpha_{ji}} \right| = \max_{i\neq j} \left| \dfrac{\operatorname{tr}(K)}{N} + \sqrt{\alpha_{ij}\alpha_{ji}} \right|.
\end{align}
Here, the second equality holds adopting the standard convention for square roots of complex numbers \( \alpha + i\beta \) are expressed as \( \pm(a + ib) \), with \( a \geq 0 \).

	\begin{thm}\label{theorem3point5twoerasure}
Let \( \mathcal{H}_n \) be a real Hilbert space of dimension \( n \) and let \( K \in \mathcal{B}(\mathcal{H}_n) \) be a positive semi-definite operator. Then the minimal spectral radius over all $K-$dual pairs for two erasures satisfies
\[
\tilde{r}_{2} \geq 
\begin{cases}
\dfrac{\operatorname{tr}(K)}{N} + \sqrt{\dfrac{N\,\operatorname{tr}(K^2) - \left(\operatorname{tr}(K)\right)^2}{N^2(N-1)}}, & \text{if } \operatorname{tr}(K^2) \geq \dfrac{\left(\operatorname{tr}(K)\right)^2}{N},\\~\\
\sqrt{\dfrac{(N-2)\left(\operatorname{tr}(K)\right)^2 + N\,\operatorname{tr}(K^2)}{N^2(N-1)}}, & \text{if } \operatorname{tr}(K^2) < \dfrac{\left(\operatorname{tr}(K)\right)^2}{N}.
\end{cases}
\]

Moreover, if there exists a  $2-$uniform \( K-\)dual pair, then the inequality above becomes an equality:
\begin{equation}\label{equation7expression}
\tilde{r}_{2} = 
\begin{cases}
\dfrac{\operatorname{tr}(K)}{N} + \sqrt{\dfrac{N\,\operatorname{tr}(K^2) - \left(\operatorname{tr}(K)\right)^2}{N^2(N-1)}}, & \text{if } \operatorname{tr}(K^2) \geq \dfrac{\left(\operatorname{tr}(K)\right)^2}{N},\\~\\
\sqrt{\dfrac{(N-2)\left(\operatorname{tr}(K)\right)^2 + N\,\operatorname{tr}(K^2)}{N^2(N-1)}}, & \text{if } \operatorname{tr}(K^2) < \dfrac{\left(\operatorname{tr}(K)\right)^2}{N}.
\end{cases}
\end{equation}
In this case, a dual pair \( (F, G) \in \mathcal{R}_2 \) if and only if it is a  $2-$uniform \( K \)-dual pair.
\end{thm}

	\begin{proof}
		Suppose \( (F, G) \) is an \( (N, n) \) dual pair. Taking \( \alpha_{ij} = \langle g_i, f_j \rangle \), it is easy to see that
\[
\sum_{i,j = 1}^N \alpha_{ij} \alpha_{ji} 
= \sum_{i=1}^N \left\langle g_i , \sum_{j=1}^N \langle f_i , g_j \rangle f_j \right\rangle 
= \sum_{i=1}^N \langle g_i , K f_i \rangle 
= \operatorname{tr}(\Theta_G^* \Theta_{K F}) 
= \operatorname{tr}\left({K^*}^2\right) 
= \operatorname{tr}(K^2).
\]
 Further, if $(F,G)$ is $1-$uniform $K-$dual pair, then
		\begin{align}\label{eqation7}
			\displaystyle{\sum_{i \neq j}\alpha_{ij}\alpha_{ji}} = tr\left({K^*}^2  \right)-\sum\limits_{i=1}^N \alpha_{ii}^2 = tr\left({K}^2  \right) - \frac{(tr(K))^2}{N}.
		\end{align}	
		Let  $\mu :=  tr\left({K}^2  \right) - \dfrac{(tr(K))^2}{N}.$ In the following, we derive a lower bound for \( \tilde{r}_{2} \), which varies according to the sign of \( \mu \).

		\par Suppose $\mu \geq 0.$ 	For $(F,G)\in \mathcal{R}_{1}, $   as in (\ref{eqn3point4*}), we have	$r_{2} (F,G) = \max\limits_{i \neq j} \left| \frac{tr(K)}{N}+ \sqrt{ \alpha_{ij}\alpha_{ji}} \right|.$
		Therefore, there exists $i_0 \neq j_0$ such that $\alpha_{i_0 j_0} \alpha_{j_0 i_0} \geq \dfrac{\sum\limits_{i \neq j}\alpha_{ij} \alpha_{ji}}{N(N-1)} = \dfrac{\mu}{N(N-1)} \geq 0,$ by Theorem \ref{thm5point2} and \eqref{eqation7}. Therefore, we obtain
\[
r_{2}(F,G) \geq \left| \dfrac{\operatorname{tr}(K)}{N} + \sqrt{ \alpha_{i_0 j_0} \alpha_{j_0 i_0} } \right| \geq \dfrac{\operatorname{tr}(K)}{N} + \sqrt{ \dfrac{\mu}{N(N-1)} } = \dfrac{\operatorname{tr}(K)}{N} + \sqrt{ \dfrac{N\,\operatorname{tr}(K^2) - \left( \operatorname{tr}(K) \right)^2 }{N^2(N-1)} }.
\]
Since the right-hand side of the inequality does not depend on the choice of the dual pair \( (F,G) \), it gives a lower bound for \( \tilde{r}_{2} \). On the other hand, suppose $\mu <0.$ Then, by \eqref{eqation7}, for any $K-$dual pair $(F,G),\;$ $\min\limits_{i \neq j} \alpha_{ij}\alpha_{ji} \leq \dfrac{\mu}{N(N-1)}.$  Therefore, there exists $i_1 \neq j_1$ such that $\alpha_{i_1 j_1} \alpha_{j_1 i_1} \leq \dfrac{\sum\limits_{i \neq j}\alpha_{ij} \alpha_{ji}}{N(N-1)} = \dfrac{\mu}{N(N-1)} \leq 0,$ by Theorem \ref{thm5point2} and \eqref{eqation7}. So, 
$$r_{2} (F,G) \geq  \left| \dfrac{tr(K)}{N} + \sqrt{ \alpha_{i_1 j_1}\alpha_{j_1 i_1}} \right| =\sqrt{\left(\dfrac{tr(K)}{N}\right)^2 - \alpha_{i_1 j_1}\alpha_{j_1 i_1}}\geq \sqrt{\left(\dfrac{tr(K)}{N}\right)^2 - \dfrac{\mu}{N(N-1)} }  = \sqrt{\dfrac{(tr(K))^2 -  tr(K^2)}{N(N-1)}}.$$

		\par 	Assume that there exists a $2$-uniform $K$-dual pair $(F'', G'')$. By applying equation~\eqref{eqn3point4*}, we obtain
\[
r_2(F'', G'') = \left| \dfrac{\operatorname{tr}(K)}{N} + \sqrt{c_{F'', G''}} \right|,
\]
where \( c_{F'', G''} \) denotes the constant arising from the $2$-uniform property of the pair \( (F'', G'') \). This constant also satisfies the identity
\[
\sum_{i \neq j} \alpha''_{ij} \alpha''_{ji} = N(N-1) c_{F'', G''},
\]
which leads to
\[
c_{F'', G''} = \dfrac{\mu}{N(N-1)}.
\]
Based on the sign of \( \mu \), we obtain the following expression for \( \tilde{r}_2 \):
\[
\tilde{r}_2 =
\begin{cases}
\dfrac{\operatorname{tr}(K)}{N} + \sqrt{\dfrac{N\,\operatorname{tr}(K^2) - (\operatorname{tr}(K))^2}{N^2(N-1)}} & \text{if } \operatorname{tr}(K^2) \geq \dfrac{(\operatorname{tr}(K))^2}{N}, \\[2em]
\sqrt{\dfrac{(N-2)(\operatorname{tr}(K))^2 + N\,\operatorname{tr}(K^2)}{N^2(N-1)}} & \text{if } \operatorname{tr}(K^2) < \dfrac{(\operatorname{tr}(K))^2}{N}.
\end{cases}
\]
From this, it can be concluded that \eqref{equation7expression} holds.
\par From the arguments above, we infer that any  $2-$uniform $K-$dual pair $(F,G)$ will satisfy ${r}_{2} (F,G) = \tilde{r}_{2}$ and hence will belong to $\mathcal{R}_{2},$ by Theorem \ref{thm5point2}.	 Conversely, if $(F,G) \in \mathcal{R}_{2},$ then it is a $1-$uniform $K-$dual pair, as it is also in $\mathcal{R}_{1}.$ We shall show that $\alpha_{k\ell}\alpha_{\ell k} = \dfrac{\mu}{N(N-1)},\,\forall\;k \neq \ell.$ Let us first consider the case when  $\mu \geq 0.$ Suppose for some $k \neq \ell,\, \alpha_{k\ell}\alpha_{\ell k} >  \dfrac{\mu}{N(N-1)}.$ Then, using \eqref{eqn3point4*},
		$${r}_{2} (F,G) \geq  \left| \dfrac{tr(K)}{N}+ \sqrt{ \alpha_{k\ell}\alpha_{\ell k}}   \right| > \dfrac{tr(K)}{N} + \sqrt{ \frac{\mu}{N(N-1)}} = \tilde{r}_{2},$$
		which contradicts the fact that $(F,G) \in \mathcal{R}_{2}.$ Therefore, $ \alpha_{k\ell}\alpha_{\ell k} \leq \dfrac{\mu}{N(N-1)},$	for all $k \neq \ell. $ If $\alpha_{k_0 \ell_0}\alpha_{\ell_0 k_0} < \dfrac{\mu}{N(N-1)},$ for some $k_0 \neq \ell_0,$ then $\mu =\displaystyle{\sum_{k \neq \ell}\alpha_{k \ell}\alpha_{\ell k} < \frac{\mu}{N(N-1)} N(N-1)} = \mu,$ thereby proving our claim. A similar argument can be applied to handle the case when \( \mu < 0 \).
	\end{proof}
	
	\noindent A proposition analogous to Theorem \ref{thm3point7} states as follows.
	\begin{prop}
Let \( U \) be a unitary operator on \( \mathcal{H}_n \) satisfying \( KU = UK \). Then, for any \((N,n)\) \( K \)-dual pair \( (F, G) \), we have \( (F, G) \in \mathcal{R}_i \) if and only if \( (UF, UG) \in \mathcal{R}_i \), for \( i = 1, 2 \).
\end{prop}

	%However, we have
%	\begin{align*}
%		r_{1}^q (F,S_{F}^{-1}F) = \max\limits_{1 \leq i \leq n} q_i \left| \langle f_i  ,S_{F}^{-1}f_i \rangle  \right| = \max\limits_{1 \leq i \leq n} q_i  \| S_{F}^{-\frac{1}{2}}f_i \|^2 = \max\limits_{1 \leq i \leq n} q_i.
%	\end{align*}
%	Hence, $\tilde{r}_{1}^q = \max\limits_{1 \leq i \leq n} q_i$ and $(F,S_{F}^{-1}F),$ with $F$ being any frame, is a $1-$erasure probabilistic spectrally optimal dual pair. Further,
%	\begin{align*}
%		&r_{2}^q (F,S_{F}^{-1}F) \\& = \frac{1}{2}\max\limits_{i \neq j} \left| q_i\langle S_{F}^{-1}f_i,f_i \rangle +  q_j\langle S_{F}^{-1}f_j,f_j \rangle \pm \sqrt{\left( q_i\langle S_{F}^{-1}f_i,f_i \rangle -  q_j\langle S_{F}^{-1}f_j,f_j \rangle\right)^2 + 4q_i q_j  \langle S_{F}^{-1}f_i,f_j \rangle \langle S_{F}^{-1}f_j,f_i \rangle} \right|
%		\\&= \frac{1}{2}\max\limits_{i \neq j} \left| q_i \|S_{F}^{-\frac{1}{2}}f_i \|^2 +  q_j\|S_{F}^{-\frac{1}{2}}f_j \|^2 \pm \sqrt{\left( q_i\|S_{F}^{-\frac{1}{2}}f_i \|^2-  q_j\|S_{F}^{-\frac{1}{2}}f_j \|^2\right)^2 + 4q_i q_j \left|\langle S_{F}^{-\frac{1}{2}}f_i,S_{F}^{-\frac{1}{2}}f_j \rangle \right|^2 } \right|
%		\\&= \frac{1}{2}\max\limits_{i \neq j} \left| q_i+  q_j \pm \sqrt{( q_i-  q_j)^2   } \right|
%		\\&= \max\limits_{1\leq i \leq n} q_i,
%	\end{align*}
%	from which it can be concluded that  $\tilde{r}_{2}^q = \max\limits_{1 \leq i \leq n} q_i$ and for any frame $F,\,(F,S_{F}^{-1}F)$ is also a $2-$erasure probabilistic spectrally optimal dual pair.
%	

	\section{\textbf{Optimal $K-$duals associated with a given $K-$frame}}
    
	In the preceding section, we examined the notion of optimality by considering all possible $K-$dual pairs in the Hilbert space $\mathcal{H}_n$. In contrast, the present section focuses on a fixed $K-$frame and explores the optimality by varying over all its possible $K-$duals. As before, we investigate this optimality with respect to two different error measures, namely the operator norm and the spectral radius.

	\subsection{Optimal $K-$duals with respect to operator norm}
	
Let $K \in \mathcal{B}(\mathcal{H}_n)$. Given a frame $F = \{f_i\}_{i=1}^N$ for $\mathcal{H}_n$, we define
\[
\widetilde{\mathbb{O}}_{1}(F) := \inf \left\{ \mathbb{O}_{1}(F, G) : G \text{ is a $K-$dual of } F \right\},
\]
A dual frame \( G' \) of \( F \) is said to be a \textit{$1-$erasure optimal $K$-dual of $F$ under the operator norm} if
\[
\mathbb{O}_{1}(F, G') = \widetilde{\mathbb{O}}_{1}(F),
\]
where \( \mathbb{O}_{1}(F, G) \) is defined as in Subsection~3.1.

The following proposition gives an adequate condition ensuring that the canonical $K-$dual of $F$ attains this optimality. To that end, define
\[
L := \max \left\{ \|f_i\| \cdot \|K^\dag f_i\| : 1 \leq i \leq N \right\},
\]
and set
\[
I_1 := \left\{ i : \|f_i\| \cdot \|K^\dag f_i\| = L \right\}, \quad I_2 := \{1,2,\ldots,N\} \setminus I_1.
\]
Let $V_j := \mathrm{span}\{ f_i : i \in I_j \}$ for $j = 1,2$. The sufficient condition for optimality will be expressed in terms of these subspaces.

\begin{prop}\label{prop4point1}
Let $F = \{f_i\}_{i=1}^N$ be a Parseval $K-$frame for the Hilbert space $\mathcal{H}_n$, where $K \in \mathcal{B}(\mathcal{H}_n)$. Suppose the subspaces $V_1$ and $V_2$, be as before, satisfy $V_1 \cap V_2 = \{0\}.$ Then the canonical $K-$dual of $F$ achieves $1-$erasure optimality with respect to the operator norm.
\end{prop}

	\begin{proof}
		We proceed by proof of contradiction. Suppose there exists a $K-$dual $G = \{g_i\}_{i=1}^N $ of $F$  such that  $\mathbb{O}_{1}(F,G) \leq \mathbb{O}_{1} (F, K^\dag F) .$ Then by \eqref{equation4}, for $i \in I_1,\;	 \|f_i \|\,\|g_i \| \leq \max\limits _{1\leq i \leq N} \|f_i \|\,\|g_i \| \leq L =  \|f_i \| \left\| K^\dag f_i \right\| .$ We can express the dual $G$ as $ \{ K^\dag f_i + u_i \}_{i=1}^N ,$  where $\sum\limits_{i=1}^N \langle f,u_i \rangle f_i =0,$ for all $f \in \mathcal{H}_n.$
		So, we may write $\|K^\dag f_i + u_i \|^2 \leq \|K^\dag f_i \|^2,$ for all $i \in I_1,$ which in turn leads to
		\begin{align}\label{eqn4point9}
			\sum_{i \in I_1} \| u_i \|^2 + 2 Re \bigg( \sum_{i \in I_1} \langle  K^\dag f_i ,u_i \rangle \bigg) \leq 0 .
		\end{align}
		Now, the condition $\sum\limits_{i=1}^N  \langle f,u_i \rangle f_i =0,\,f \in \mathcal{H}_n$ can be rewritten as $\sum\limits_{i \in I_1}  \langle f,u_i \rangle f_i + \sum\limits_{i \in I_2}  \langle f,u_i \rangle f_i = 0.$ As $V_1 \cap V_2 = \{0\},$  we get $\sum\limits_{i \in I_1} \langle  f , u_i \rangle f_i = 0, \text{for all} \;f \in \mathcal{H}_n .$ In other words, taking $U_1 = \{ u_i\}_{ i \in I_1}$ and  $F_1 = \{ f_i\}_{ i \in I_1},$ we have  $\Theta_{F_1}^*\Theta_{U_1} = 0.$ So, $\sum\limits_{i \in I_1} \langle  K^\dag f_i ,u_i \rangle =  tr\left(\Theta_{U_1} \Theta_{K^\dag F_1}^*  \right) = tr\left( \Theta_{K^\dag F_1}^*\Theta_{U_1}  \right) = tr \left(K^\dag \Theta_{ F_1}^*\Theta_{U_1}  \right)= 0.$  Thus, from \eqref{eqn4point9}, we have $ \sum\limits_{i \in I_1} \| u_i \|^2 \leq  0,$ which gives $u_i =0,\,\forall i \in I_1.$ Therefore, $\mathbb{O}_{1}(F,G) \geq \max\limits_{i \in I_1}\|f_i \|\,\|g_i \| = \max\limits_{i \in I_1}\|f_i \|\,\|k^\dag f_i \| = L = \mathbb{O}_{1}(F,K^\dag F).$  In other words, $\mathbb{O}_{1}(F,G) = \mathbb{O}_{1}(F,K^\dag F)$ and hence $K^\dag F$ is a $1-$erasure  optimal $K-$dual of $F.$
		
	\end{proof}

\begin{thm}\label{theorem4point3op}
Let $F = \{f_i\}_{i=1}^N$ be a Parseval $K-$frame for $\mathcal{H}_n$, where $K$ is a positive semi-definite operator on $\mathcal{H}_n$. Suppose that $V_1 \cap V_2 = \{0\}$ and that the subset $\{f_i\}_{i \in \Lambda_1}$ is linearly independent. Then, the canonical $K-$dual of $F$ is a $1-$erasure optimal $K-$dual of $F$ with respect to the operator norm. Moreover, if $N > n$, the collection of all $1-$erasure optimal $K-$duals of $F$ under the operator norm is uncountable.
\end{thm}

\begin{proof}
	First part follows directly from Proposition\ref{prop4point1}.
	\par  Suppose $N>n.$ Then there exists a $K-$dual $G$ of $F$ other than the canonical dual. In other words, there exists a $K-$dual $G=\{g_i\}_{i=1}^N = \{K^\dag f_i +u_i\}_{i=1}^N$ of $F$ such that $u_i = 0,$ for all $i \in \Lambda_1$ and $U_{i_0} \neq 0$ for some $i_0 \in \Lambda_2.$ As $t \to \|f_i\|\,\|K^\dag f_i + tu_i\|$ is a  continuous  function on $\mathbb{R}$ and $\|f_i\|\,\|K^\dag f_i + u_i\|<L,$ for all $i \in \Lambda_2,$ there exists a $\delta_{i}^K>0$ such that for any $0 \leq |t| \leq \delta_{i}^K,\,$ $\|f_i\|\,\|K^\dag f_i + t u_i\| < L,$ for any $i \in \Lambda_2.$ Taking $\delta= \min \{\delta_{i}^K : i \in \Lambda_2\}$ we have, for all $i \in \Lambda_2,\,0 <  |t| < \delta,\,\|f_i\|\,\|K^\dag f_i + t u_i\| < L.$ Let $|t| \leq \delta.$ Clearly, $G_t = \{K^\dag f_i +u_i\}_{i=1}^N$ is a $K-$dual of $F.$ Further,
	\begin{align*}
		\mathbb{O}_{1} (F, G_t) &= \max_{1\leq i \leq N} \|f_i\|\,\|K^\dag f_i + t u_i\| \\&= L \\&=  \max_{1\leq i \leq N} \|f_i\|\,\|K^\dag f_i \| \\&= 	\mathbb{O}_{1} (F, K^\dag F) \\&=	\widetilde{\mathbb{O}}_{1} (F)
	\end{align*}
	Therefore, $G_t = \{K^\dag f_i +u_i\}_{i=1}^N$ is a $1-$erasure optimal $K-$dual of $F$  with respect to the operator norm, for $|t| \leq \delta.$
\end{proof}

\begin{rem}
    Note that if $N=n$ then the canonical dual $K^\dag F$ is the only $K-$dual of $F.$  Hence, the canonical dual is the unique $1-$erasure optimal $K-$dual of $F$ under operator norm.
\end{rem}

\begin{lem}\label{lem3point6}
	Let $\{f_1,\ldots,f_m\} \subset \mathcal{H}_n$ be linearly independent and $\alpha $ be a scalar. Then, there exists $h \in \mathcal{H}_n$ such that $\langle f_i,h \rangle = \alpha, \;\forall\,i=1,2,\ldots,m.$
\end{lem}
\begin{proof}
	For $1 \leq i \leq m,$ $f_i$ can be written as $f_i = f_{i1}e_1 + \cdots + f_{in}e_n,$ where $\{e_1,\ldots,e_n\}$ is an orthonormal basis for $\mathcal{H}_n.$ Consider the following system of $m$ equations in $n$ unknowns, namely, ${h}_1,{h}_2,\ldots,{h}_n.$ 
	\begin{align*}\label{eqn7}
		&f_{11} {h}_1 + \cdots +  f_{1n} {h}_n = \alpha\\
		&f_{21} {h}_1 + \cdots + f_{2n} {h}_n= \alpha\\
		&\qquad\qquad\vdots\\
		&f_{m1} {h}_1 + \cdots + f_{mn} {h}_n = \alpha.
	\end{align*}
	As $\{f_1,\ldots,f_m\}$ is linearly independent, the coefficient matrix of the above system has rank $m$ and so does the augmented matrix. Hence, the above system of equations has a solution and $h = \sum \limits_ {j=1}^n h_j e_j$ satisfies  $\langle f_i,h \rangle = \alpha \;\forall\,i=1,2,\ldots,m.$ 
\end{proof}

The following theorem characterizes the conditions under which the canonical dual of a Parseval $K-$frame is not only optimal with respect to the operator norm for a single erasure but also unique. This result establishes a geometric criterion, specifically involving the intersection of certain subspaces and the linear independence of a subset of frame elements that guarantees the uniqueness of this optimal dual.
\begin{thm}
Let $F$ be a Parseval $K-$frame for $\mathcal{H}_n$. Then the following statements are equivalent:
\begin{enumerate}
    \item[\emph{(i)}] The canonical dual $\{K^\dag f_i\}_{i=1}^N$ is the unique $1-$erasure optimal $K-$dual of $F$ under the operator norm.
    \item[\emph{(ii)}] The subspaces $V_1$ and $V_2$ satisfy \( V_1 \cap V_2 = \{0\} \) and the set \( \{f_i\}_{i \in \Lambda_2} \) is linearly independent.
\end{enumerate}
\end{thm}

\begin{proof}
	$(i)\implies (ii)$ First we prove $\{f_i\}_{i \in \Lambda_2}$ is linearly independent. Suppose not. Then there exists $\{c_i\}_{i \in \Lambda_2},$ not all zero such that $\sum\limits_{i \in \Lambda_2}c_i f_i =0.$ Let $\widetilde{\Lambda_2} = \{i \in \Lambda_2 : c_i \neq 0\}.$ Take a non-zero element $u$ in $\mathcal{H}_n.$  Taking $U =\{u_i\}_{i=1}^N$ as
	\begin{align*}
		u_i = \begin{cases}
			u\bar{c_i}, \;\;& for \;i \in \widetilde{\Lambda_2},\\
			0,\;\; &\text{otherwise},
		\end{cases}
	\end{align*}

Obviously, $\sum\limits_{i=1}^N \langle f,f_i\rangle u_i = \sum\limits_{i \in \widetilde{\Lambda_2}} \langle f,f_i\rangle u_i + \sum\limits_{i \in \widetilde{\Lambda_2}^c} \langle f,f_i\rangle u_i =  \left\langle f,\sum\limits_{i \in \widetilde{\Lambda_2}} c_if_i\right\rangle u =0$ and hence $\{K^\dag f_i + U_i\}_{i=1}^N$ is a $K-$dual of $F.$ For all $i \in \Lambda_2,\,\|K^\dag f_i\|\,\|f_i\| <L.$ Then there exist a $\delta >0$ such that  $i \in \Lambda_2,\,\|K^\dag f_i + tu_i\|\,\|f_i\| <L, $ for all $t \in (-\delta, \delta),$ as analogous to Theorem \ref{theorem4point3op}. Then for the dual $G_t = \{K^\dag f_i + tu_i\}_{i=1}^N,\,$ $\mathbb{O}_{1} (F, G_t) = L$ and hence $G_t$ is a $1-$erasure optimal $K-$dual of $F$ under operator norm for any $t \in  (-\delta, \delta),$ which gives a contradiction. \\
Now we show that $V_1 \cap V_2 = \{0\}.$ If not, then there exist a $0 \neq f \in V_1 \cap V_2.$
Then, $f$ can be written as $f= \sum\limits_{j \in \Lambda'_1 \subset \Lambda_1} c_{j} f_{j} = \sum\limits_{k \in \Gamma_2} d_k f_k, $ where $c_j \neq 0$ and $\{f_{j}\}_{j \in \Lambda_1'}$  is also linearly independent. As $\{ c_j f_j\}_{j \in \Gamma'_1}$ is also linearly  independent, by Lemma \ref{lem3point6}, there exists  $h\in \mathcal{H}_n$ such that $\langle  f_j, \bar{c}_j h \rangle <0,$ for all $ j \in \Lambda'_1.$ Now, let  $\tilde{U} =\{\tilde{u}_i\}_{i=1}^N,$ where
\[	\tilde{u}_i=\begin{cases}
	\bar{c}_i h, & \text{if $i \in \Lambda'_1$}\\
	0, & \text{if $i \in \Lambda_1 \setminus \Lambda'_1$}\\
	-\bar{d}_i h, & \text{if $i \in \Lambda_2$}.
\end{cases}
\] 
Then for any $t\in \mathbb{R},\, \sum\limits_{1\leq i \leq N} \langle f,f_i \rangle t \Tilde{u}_i  =\sum\limits_{i \in \Lambda'_1} \langle f,c_if_i \rangle th - \sum\limits_{i \in \Lambda_2} \langle f,d_if_i \rangle th = \left\langle f, \sum\limits_{i \in \Lambda'_1}c_i f_i - \sum\limits_{i \in \Lambda_2} d_i f_i \right\rangle th =0.$ Therefore, $ \{K^\dag f_i + t\Tilde{u}_i\}_{i=1}^N$ is a noncanonical $K-$dual of $F,$ for any $t \in \mathbb{R}.$ For all $i \in \Lambda_2,\,$ $ \|f_i\|\left\|K^\dag f_i  \right\|  < L.$ Then there exists $\tilde{t}_i >0$ such that $ \|f_i\|\left\|K^\dag f_i + \tilde{t}_i \Tilde{u}_i \right\|  < L,$ for each $i \in \lambda_2.$ Taking  $ \tilde{t}_0 = \min \{t_i : i \in \Lambda_2\},$ we have 
$$   \|f_i\|\left\|\frac{1}{A}f_i + \tilde{t}_0 \Tilde{u}_i \right\|  < L, \text{ for all } i \in \Lambda_2. $$
Further, for $t \in \mathbb{R},\,i \in \Lambda_1 \setminus \Lambda'_1 , \,$ we have $  \|f_i\| \left\|K^\dag f_i +t\Tilde{u}_i \right\|  = L$ and for   $i \in \Lambda'_1 ,$ 
\begin{align*}
\|f_i\|^2\left\|K^\dag f_i +t\Tilde{u}_i \right\|^2  &=  \|f_i\|^2 \left( \|K^\dag f_i\|^2 + t^2\|\Tilde{u}_i\|^2 + 2t Re\langle K^\dag f_i,\Tilde{u}_i\rangle \right) \\&= L^2 +  t \|f_i\|^2 \left(  t\|\Tilde{u}_i\|^2 + 2 Re\langle K^\dag f_i,\Tilde{u}_i\rangle\right) .
\end{align*}

\normalsize
The negative value of $\langle  f_i, \Tilde{u}_i \rangle $ allows us to choose $t'_i >0$ small enough such that $\left(  t_i \|\Tilde{u}_i\|^2 + 2 Re\langle K^\dag f_i,\Tilde{u}_i\rangle\right)< 0$ and hence $\|f_i\|^2\left\|K^\dag f_i +t_i \Tilde{u}_i \right\|^2 <L^2, $ for all $i \in \Lambda'_1 .$  Furthermore, taking $\tilde{t}_1 = \min \{t_i : i \in \Lambda'_1\},$ we get   $$\|f_i\|\,\left\|K^\dag f_i +\tilde{t}_1 \Tilde{u}_i \right\|  < L, \text{ for all } i \in \Gamma'_1.$$
If we choose  $\Tilde{t} = min \{ \tilde{t}_0,t_1\},$ we obtain a dual $\Tilde{G}= \left\{ K^\dag f_i + \Tilde{t}\Tilde{u}_i\\\\\right\}_{i=1}^N$ of $F$ for which $r^{(1)}(F,\Tilde{G}) = L,$ thereby contradicting the hypothesis.\\

$(ii) \implies (i)$   Assume that $V_1 \cap V_2 = \{0\}$ and  $\{f_i\}_{i \in \Lambda_2}$ is linearly independent. By Theorem [...],  $K^\dag F $ is a $1-$erasure optimal $K-$dual of $F$ under operator norm. Suppose $G= \left\{ K^\dag f_i +u_i \right\}_{i=1}^N $  is a $1-$erasure optimal $K-$dual of $F$ under operator norm  wherein $\displaystyle{ \sum_{i=1}^N \langle f,u_i \rangle f_i = 0} ,\; \text{for all }f \in \mathcal{H}_n.$ Further, using the condition $V_1 \cap V_2 = \{0\},$ we obtain 
\begin{align}\label{eqn4.10}
	\sum\limits_{i \in  \Lambda_1 } \langle f,u_i \rangle f_i =0 =\sum\limits_{i \in  \Lambda_2 } \langle f,u_i \rangle f_i, \;f \in \mathcal{H}_n.
\end{align}
 This leads to $\Theta_{F_1}^* \Theta_{U_1} = 0, \; \textit{where} \; F_1 = \{f_i\}_{i \in  \Lambda_1}\; \textit{and} \;\; U_1 = \{u_i\}_{ i \in  \Lambda_1}.$ Therefore, $\Theta_{K^\dag F_1}^* \Theta_{U_1} = K^\dag \Theta_{F_1}^* \Theta_{U_1} =0 $ and so	$	tr\left(\Theta_{K^\dag F_1}^* \Theta_{U_1}\right) = tr\left( \Theta_{U_1} \Theta_{K^\dag F_1}^* \right)  =\sum\limits_{i \in  \Lambda_1 } \langle K^\dag f_i , u_i \rangle = 0.$ Consequently,
\begin{align}\label{eqn4.11}
	Re \left(\sum\limits_{i \in  J_1 } \langle K^\dag f_i , u_i \rangle \right) = 0.
\end{align}

 For $i \in \Lambda_1,\;$ $\max\limits_{i \in \Lambda_1} \|f_i\|\left\|K^\dag f_i + u_i\right\| \leq \max\limits_{1 \leq i \leq N} \|f_i\|\left\|K^\dag f_i + u_i\right\| = L = \max\limits_{i \in \Lambda_1} \|f_i\|\left\|K^\dag f_i \right\|. $ This leads to, $\|f_i\|\left\|K^\dag f_i + u_i\right\| \leq \|f_i\|\left\|K^\dag f_i \right\|,\,\forall i \in \Lambda_1.  $ This gives $\left\|K^\dag f_i + u_i\right\|^2 \leq \left\|K^\dag f_i \right\|^2,\,\forall i \in \Lambda_1. $ Consequently, $\|u_i\|^2 + 2 Re\langle K^\dag f_i, u_i\rangle \leq 0,$  for all $ i \in \Lambda_1.$ Therefore, $\sum\limits_{i \in \Lambda_1}\|u_i\|^2 + 2 Re \left(\sum\limits_{i \in \Lambda_1}\langle K^\dag f_i, u_i\rangle \right)\leq 0$ and hence using \eqref{eqn4.11}, we have 
$\sum\limits_{i \in \Lambda_1}\|u_i\|^2 \leq 0.$ This suggest that $u_i =0$  for all $ i \in \Lambda_1.$ Also using \eqref{eqn4.10} and  the fact  $\{f_i\}_{i \in \Lambda_2}$ is linearly independent we have $u_i =0$  for all $ i \in \Lambda_2.$ Thus, $u_i =0, 1 \leq i \leq N.$ 
\end{proof}

\begin{cor}
Let $F = \{f_i\}_{i=1}^N$ be a Parseval $K$-frame for $\mathcal{H}_n$. If  $\|f_i\|\,\|K^\dag f_i\|$ is constant for all $1 \leq i \leq N$, then the canonical $K-$dual $K^\dag F$ of $F$ is the unique $1-$erasure optimal $K-$dual of $F$ under operator norm.
\end{cor}

The following proposition provides a condition under which the canonical $K-$dual fails to be $1-$ erasure optimal with respect to the operator norm.

\begin{prop}
Let $F = \{f_i\}_{i=1}^N$ be a Parseval $K$-frame for $\mathcal{H}_n$. Assume that the subset $\{f_i\}_{i \in \Lambda_1}$ is linearly independent and there exists a sequence $\{c_i\}_{i=1}^N$ with $c_i \neq 0$ for all $i \in \Lambda_1$ such that $\sum_{i=1}^N c_i f_i = 0.$ Then, the canonical $K-$dual $K^\dag F$ of $F$ is not a $1-$erasure optimal $K-$dual of $F$ under the operator norm.
\end{prop}

\begin{proof}
	Since $\{f_i\}_{i \in \Lambda_1}$ is linearly independent, then $\{c_i K^\dag f_i\}_{i \in \Lambda_1}$ is also linearly independent and hence there exists a $h \in \mathcal{H}_n$ such that for all $i \in \Lambda_1,$ $\langle K^\dag f_i, \bar{c}_i h \rangle = \langle c_i K^\dag f_i, h \rangle < 0,$ by Lemma\ref{lem3point6}. Let $U=\{u_i\}_{i=1}^N,$ where $u_i = \bar{c}_i h;1 \leq i \leq N. $ It is easy to see that $\sum\limits_{i=1}^N \langle f, f_i \rangle u_i = \left \langle f, \sum\limits_{i=1}^N c_i f_i \right \rangle h = 0.$ Therefore, $\{K^\dag f_i + tu_i\}_{i=1}^N$ is a $K-$dual of $F$ for all $t.$ Now, for all $i \in \Lambda_1,$
	\begin{align*}
		\|f_i\|^2\left\|K^\dag f_i +tu_i \right\|^2  &=  \|f_i\|^2 \left( \|K^\dag f_i\|^2 + t^2\|u_i\|^2 + 2t Re\langle K^\dag f_i,u_i\rangle \right) \\&= L^2 +  t \|f_i\|^2 \left(  t\|u_i\|^2 + 2 Re\langle K^\dag f_i,u_i\rangle\right).
	\end{align*}
 As $t \to \|f_i\|\,\|K^\dag f_i + tu_i\|$ is a  continuous  function on $\mathbb{R}$ then there exist $t_i >0 $ such that  $ L^2 +  t \|f_i\|^2 \left(  t\|u_i\|^2 + 2 Re\langle K^\dag f_i,u_i\rangle\right) < L^2$ for all $t \in (-t_i, t_i)$ and for each $i \in \Lambda_1.$ As for each $j \in \Lambda_2,\;$ $\|f_j\|^2\,\|K^\dag f_j \|^2 <L^2,$ then there exist $t'_j>0$ small enough such that, for any $t \in (-t'_j, t'_j),\;$  $\|f_j\|^2 \left( \|K^\dag f_j\|^2 + t^2\|u_j\|^2 + 2t Re\langle K^\dag f_j,u_j\rangle \right) < L,$ for each $j \in \Lambda_2.$ Taking $\widetilde{t} = \min\limits_{i \in \Lambda_1, j \in \Lambda_2} \left\{t_i, t'_j\right\},$ we have, for $t \in (-\widetilde{t}, \widetilde{t} ),\;$ $\|f_i\|^2\left\|K^\dag f_i +tu_i \right\|^2  =  \|f_i\|^2 \left( \|K^\dag f_i\|^2 + t^2\|u_i\|^2 + 2t Re\langle K^\dag f_i,u_i\rangle \right) < L,$ for any $1 \leq i \leq N.$ Hence, for the dual $G_t = \{K^\dag f_i + tu_i\}_{i=1}^N ,$\; $\mathbb{O}_{1}(F,G_t) <L = \mathbb{O}_{1}(F,G),\;\forall t \in  (-\widetilde{t}, \widetilde{t} ).$

\end{proof}

	\noindent We observe that the property of $1$-erasure optimality under the operator norm is preserved under unitary transformations.

\begin{rem}\label{thm4point3}
Let \( F \) be a \( K -\)frame for \( \mathcal{H}_n \), and let \( U \) be a unitary operator on \( \mathcal{H}_n \) such that \( UK = KU \). Then  \( G \) is a $1-$erasure optimal \( K-\)dual of \( F \) with respect to the operator norm if and only if \( UG \) is a $1-$erasure optimal \( K-\)dual of \( UF \) with respect to the operator norm.
\end{rem}

	\subsection{Optimal $K-$dual frames under spectral radius}
    \vspace{.2 cm}
	Let \( F = \{f_i\}_{i=1}^N \) be a \( K -\)frame for \( \mathcal{H}_n \). Let us define
\begin{align*}
	\tilde{r}_1(F) := \inf \left\{ r_1(F, G') : G' \text{ is a } K-\text{dual of } F \right\}.
\end{align*}
A dual frame \( \tilde{G} \) of \( F \) is called a \textit{$1-$erasure spectrally optimal \( K \)-dual} of \( F \) if $r_1(F, \tilde{G}) = \tilde{r}_1(F).$ For each integer \( m \) such that \( 1 < m \leq N \), we define 
\[
\tilde{r}_{m} (F) := \inf \bigg\{ r_{m} (F,G') : G' \;\text{is an } (m-1)-\text{erasure} \text{ spectrally optimal $K-$dual of}\;\; F \bigg\},
\]
where \( \mathcal{R}_{m-1}(F) \) denotes the collection of all \((m-1)\)-erasure spectrally optimal \( K \)-duals of \( F \). A dual frame \( \tilde{G} \) of \( F \) is called an \textit{\( m-\)erasure spectrally optimal \( K-\)dual} of \( F \) if it is already an \((m{-}1)\)-erasure spectrally optimal \( K-\)dual and satisfies the condition \( r_m(F, \tilde{G}) = \tilde{r}_m(F) \). The collection of all such \( K-\)duals is denoted by \( \mathcal{R}_m(F) \).

  \begin{example}
Consider the $K-$frame $F = \{f_1, f_2, f_3, f_4\}$ for $\mathbb{R}^3$, where
\[
f_1 = \begin{bmatrix} 1 \\ 0 \\ 0 \end{bmatrix}, \quad
f_2 = \begin{bmatrix} 1 \\ 0 \\ 0 \end{bmatrix}, \quad
f_3 = \begin{bmatrix} \sqrt{2} \\ 0 \\ 0 \end{bmatrix}, \quad
f_4 = \begin{bmatrix} 0 \\ 1 \\ 0 \end{bmatrix},
\]
and let
\[
K = \begin{bmatrix}
2 & 0 & 0 \\
0 & 1 & 0 \\
0 & 0 & 0
\end{bmatrix}.
\]
It is easy to verify that $F$ is a Parseval $K$-frame for $\mathbb{R}^3$. The Moore–Penrose pseudoinverse of $K$ is
\[
K^\dag = \begin{bmatrix}
\frac{1}{2} & 0 & 0 \\
0 & 1 & 0 \\
0 & 0 & 0
\end{bmatrix}.
\]
Hence, the canonical $K$-dual of $F$ is given by
\[
K^\dag F = \left\{
\begin{bmatrix} \frac{1}{2} \\ 0 \\ 0 \end{bmatrix},
\begin{bmatrix} \frac{1}{2} \\ 0 \\ 0 \end{bmatrix},
\begin{bmatrix} \frac{1}{\sqrt{2}} \\ 0 \\ 0 \end{bmatrix},
\begin{bmatrix} 0 \\ 1 \\ 0 \end{bmatrix}
\right\}.
\]

We compute:
\[
\|f_1\| \cdot \|K^\dag f_1\| = \|f_2\| \cdot \|K^\dag f_2\| = \frac{1}{2}, \quad 
\|f_3\| \cdot \|K^\dag f_3\| = \|f_4\| \cdot \|K^\dag f_4\| = 1.
\]
Thus, $\mathbb{O}_{1}(F, K^\dag F) = 1$.

Any $K-$dual $G = \{g_i\}_{i=1}^4$ of $F$ must be of the form:
\[
\begin{aligned}
g_1 &= \begin{bmatrix} \frac{1}{2} + \alpha_1 \\ \alpha_2 \\ \alpha_3 \end{bmatrix}, \quad
g_2 = \begin{bmatrix} \frac{1}{2} + \beta_1 \\ \beta_2 \\ \beta_3 \end{bmatrix}, 
g_3 &= \begin{bmatrix} \frac{1}{\sqrt{2}} - \frac{\alpha_1 + \beta_1}{\sqrt{2}} \\ -\frac{\alpha_2 + \beta_2}{\sqrt{2}} \\ -\frac{\alpha_3 + \beta_3}{\sqrt{2}} \end{bmatrix}, \quad
g_4 = \begin{bmatrix} 0 \\ 1 \\ 0 \end{bmatrix}, \quad \alpha_i, \beta_i \in \mathbb{R}.
\end{aligned}
\]

We now compute:
\begin{align*}
\|f_1\| \cdot \|g_1\| &= \sqrt{\left(\frac{1}{2} + \alpha_1\right)^2 + \alpha_2^2 + \alpha_3^2}, \\
\|f_2\| \cdot \|g_2\| &= \sqrt{\left(\frac{1}{2} + \beta_1\right)^2 + \beta_2^2 + \beta_3^2}, \\
\|f_3\| \cdot \|g_3\| &= \sqrt{(1 - \alpha_1 - \beta_1)^2 + (\alpha_2 + \beta_2)^2 + (\alpha_3 + \beta_3)^2}, \\
\|f_4\| \cdot \|g_4\| &= 1.
\end{align*}

Hence, for any $K$-dual $G$ of $F$, we have $\mathbb{O}_{1}(F, G) \geq 1$, and thus the canonical dual $K^\dag F$ is a $1$-erasure optimal $K$-dual of $F$ under operator norm.

Now, take $\alpha_1 = \beta_1 = 0.05$, and $\alpha_2 = \alpha_3 = \beta_2 = \beta_3 = 0$. Then the corresponding dual is:
\[
G' = \left\{
\begin{bmatrix} 0.55 \\ 0 \\ 0 \end{bmatrix}, 
\begin{bmatrix} 0.55 \\ 0 \\ 0 \end{bmatrix}, 
\begin{bmatrix} 0.6364 \\ 0 \\ 0 \end{bmatrix}, 
\begin{bmatrix} 0 \\ 1 \\ 0 \end{bmatrix}
\right\},
\]
for which $\mathbb{O}_{1}(F, G') = 1$. Hence, the canonical dual is \emph{not} the unique $1$-erasure optimal $K$-dual under the operator norm.

We now compute:
\[
\langle f_1, K^\dag f_1 \rangle = \langle f_2, K^\dag f_2 \rangle = \frac{1}{2}, \quad 
\langle f_3, K^\dag f_3 \rangle = \langle f_4, K^\dag f_4 \rangle = 1.
\]
Thus, $r_1(F, K^\dag F) = 1$.

For any $K$-dual $G$ of $F$, we have:
\begin{align*}
\langle f_1, g_1 \rangle &= \frac{1}{2} + \alpha_1, \\
\langle f_2, g_2 \rangle &= \frac{1}{2} + \beta_1, \\
\langle f_3, g_3 \rangle &= 1 - \alpha_1 - \beta_1, \\
\langle f_4, g_4 \rangle &= 1.
\end{align*}
Thus, $r_1(F, G) \geq 1$. Therefore, the canonical dual $K^\dag F$ is a $1$-erasure spectrally optimal $K$-dual of $F$. Moreover, for the dual $G'$ above, we also have $r_1(F, G') = 1$. Therefore, the canonical dual fails to be the unique $1-$erasure spectrally optimal $K-$dual of $F$.

Finally, note that:
\[
\frac{\operatorname{tr}(K)}{N} = \frac{3}{4}.
\]
Thus there is no dual $G$ for which $(F, G)$ is a $(4,2)$ $1-$erasure optimal $K-$dual pair under the operator norm or spectral radius.
\end{example}

\vskip 1em
Next, we examine the existence of a $1-$erasure spectrally optimal $K-$dual corresponding to a given \textit{linearly connected frame} \cite{sali}. Two vectors \( f_i \) and \( f_j \) in a frame \( F = \{f_i\}_{i=1}^N \) are said to be \textit{linearly \( F \)-connected} if there exist elements \( f_{\ell_1}, \ldots, f_{\ell_m} \in F \) such that the set \( \{f_j, f_{\ell_1}, \ldots, f_{\ell_m}\} \) is linearly independent and
\[
f_i = c f_j + \sum_{k=1}^m c_k f_{\ell_k},
\]
where \( c \) and \( c_k \) are all nonzero scalars. The frame \( F \) is called a \textit{linearly connected frame} if every pair of vectors in \( F \) is linearly \( F-\)connected.

	\begin{thm}\label{theorem4point3}
		Let $K$ be a positive semi-definite operator in $\mathcal{H}_n.$ Then, for a linearly connected Parseval $K-$frame $F,$  there always exists a $1-$erasure spectrally optimal $K-$dual frame of $F$ and  $\tilde{r}_{1} (F) = \dfrac{tr(K)}{N}.$ 
	\end{thm}
	
	\begin{proof}
		Let $G'$ be any $K-$dual of $F.$ Then, $r_{1}(F,G')\geq \dfrac{tr(K)}{N},$ as proved in Theorem \ref{thm5point2}. Let us define
		\begin{align*}
			I_{G'} = \left\{i:\langle g'_i,f_i\rangle = \frac{tr(K)}{N}\right\}\;\text{and}\;
			I_{G'}^c= \{1,2,\ldots,N\}\setminus I_{G'} .
		\end{align*}
		If the number of elements $|I_{G'}| $ in $I_{G'}$ is $N,$ then $r_{1}(F,G')= \dfrac{tr(K)}{N}.$ In other words, $\tilde{r}_{1}(F)= \dfrac{tr(K)}{N}$ and  $G'$ is a $1-$erasure spectrally optimal $K-$dual of $F.$ Interestingly, $|I_{G'}| = N-1$  is not possible for, if $(I_{G'})^c = \{j\},$ then $\sum\limits_{i=1}^N \dfrac{tr(K)}{N} = tr(K) = \sum\limits_{1 \leq i \leq N}\langle g'_i,f_i\rangle =  \sum\limits_{\substack{1 \leq i \leq N \\ i\neq j}}\langle g'_i,f_i\rangle + \langle g'_j,f_j\rangle = \sum\limits_{\substack{1 \leq i \leq N \\ i\neq j}} \dfrac{tr(K)}{N} + \langle g'_j,f_j\rangle,$ which implies that $j \in I_{G'}.$
		
		Now, we shall assume that $|I_{G'}| \leq N-2$ and prove that there exits a dual $G''$ of $F$ such that $|I_{G''}| > |I_{G'}|.$ Suppose $i_1,i_2 \in (I_{G'})^c.$ As $f_{i_1}$ and $f_{i_2}$ are linearly $F-$connected, there exist $f_{\ell_1},f_{\ell_2},\ldots,f_{\ell_m}$ in $F$ such that
		\begin{align}\label{eqn4point14}
			f_{i_1} = cf_{i_2} + \sum\limits_{1 \leq k \leq m}c_{\ell_k}f_{\ell_k},\;\;c,c_k \neq 0
		\end{align}
		and $\{f_{i_2}, f_{\ell_1},f_{\ell_2},\ldots,f_{\ell_m}\}$ is linearly independent. As, $f_{i_2} \notin span\{f_{\ell_1},f_{\ell_2},\ldots,f_{\ell_m}\},$ there exists $v\in \mathcal{H}_n$ such that $\langle v, f_{\ell_k} \rangle =0, $ for all $1 \leq k \leq m$ and $\langle v,f_{i_2} \rangle = \dfrac{1}{\bar{c}}\left(\dfrac{tr(K)}{N} - \langle g'_{i_2},f_{i_2} \rangle  \right) \neq 0.$ Taking $V =\{v_i\}_{i=1}^N$ as
		\begin{align*}
			v_i = \begin{cases}
				-v, \;\;& for \;i=i_1,\\
				\bar{c}v,\;\; & for \;i= i_2, \\
				\bar{c_{i}}v, \;\;& for \; i=\ell_k, 1 \leq k \leq m, \\
				0,\;\; &\text{otherwise},
			\end{cases}
		\end{align*}
		we get $G'' = \{g'_i + v_i\}_{i=1}^N$ as a dual of $F$ for,
		\begin{align*}
			\sum\limits_{1\leq i \leq N} \langle f, v_i \rangle f_i =  \langle f, -v \rangle f_{i_1} +  \langle f, \bar{c}v \rangle f_{i_2} +\sum\limits_{1\leq k \leq m} \langle f, \overline{c_{\ell_k}}v \rangle f_{\ell_k}   = \langle f, v \rangle \left(\sum\limits_{1\leq k \leq m} c_{\ell_k} f_{\ell_k} + cf_{i_2} - f_{i_1} \right) =0,\; f \in \mathcal{H}_n,
		\end{align*}
		by \eqref{eqn4point14}. Now, $\langle g''_{i_2},f_{i_2} \rangle = \langle g'_{i_2}+\bar{c}v,f_{i_2} \rangle = \dfrac{tr(K)}{N}$ and for $i= \ell_1,\ell_2, \ldots, \ell_m, \;  \langle g''_i,f_i \rangle = \langle g'_i + \bar{c_i}v,f_i \rangle = \langle g'_i,f_i \rangle + \bar{c_i}\langle v,f_i \rangle = \langle g'_i,f_i \rangle.$ Further, for $ i \in \{1,2,\ldots,N\}\setminus \{i_1,i_2,\ell_1,\ldots,\ell_m\},\;\langle g''_{i},f_{i} \rangle = \langle g'_{i},f_{i} \rangle.$ This suggests that $I_{G'}\subset I_{G''}$ and $i_2 \in I_{G''}.$ In other words, $|I_{G''}| > |I_{G'}|.$ By iteratively applying this procedure, we eventually obtain after at most \( N-1 \) steps, a dual \( \tilde{G} \) for which \( |I_{\tilde{G}}| = N \). This implies that \( r_1^q(F, \tilde{G}) = 1 \). Consequently, even in this scenario, we conclude that \( \tilde{r}_1(F) = \dfrac{\operatorname{tr}(K)}{N} \), and the dual \( \tilde{G} \) is indeed a $1-$erasure spectrally optimal \( K-\)dual of \( F \).

		\end{proof}

Next we want to extend the preceding theorem to an arbitrary $K-$frame \( F \) through its partitioning into linearly connected subsets. As established in \cite{sali}, any frame \( F = \{f_i\}_{i=1}^N \) for a Hilbert space \( \mathcal{H}_n \) admits a partition \( \{\Lambda_k\}_{k=1}^J \) of the index set \( \{1,2,\ldots,N\} \), unique up to permutations, such that each subset \( \{f_i\}_{i \in \Lambda_k} \), for \( k = 1,2,\ldots,J \), forms a linearly connected frame. Moreover, this decomposition induces an orthogonal direct sum of \( \mathcal{H}_n \) given by
\[
\mathcal{H}_n = \bigoplus_{i=1}^J H_i, \quad \text{where } H_i = \operatorname{span} \{f_j : j \in \Lambda_i\}.
\]

Given this decomposition of \( \mathcal{H}_n \) corresponding to a $K-$frame \( F \), we define
\[
\delta_{i}^K := \frac{\operatorname{tr}(K|_{H_i})}{|\Lambda_i|}, \quad \text{for } i = 1,2,\ldots,J,
\]
where \( \operatorname{tr}(K|_{H_i}) \) denotes the trace of the restriction of \( K \) to the subspace \( H_i \), and \( |\Lambda_i| \) is the cardinality of the corresponding index set. The sequence \( \{\delta_{i}^K\}_{i=1}^J \) is referred to as the \textit{$K-$redundancy distribution} of the frame \( F \).

	\begin{lem}\label{lemma4point9}
		Let $F = \{f_i\}_{i=1}^N $ be a $K-$frame for $\mathcal{H}_n,$ where $K$ be a positive semi-definite operator. Let $G$ be a $K-$dual of $F.$ Set $\Lambda_G = \left\{i : \langle f_i,g_i \rangle = \dfrac{tr(K)}{N}\right\}.$ Suppose $\left|\Lambda_{G}^c\right| >1$ and $i_1,i_2 \in \Lambda_{G}^c$ such that $f_{i_1}$ and $f_{i_2}$ is linearly connected. Then there exists a dual $G'$ such that $\left|\Lambda_{G'}\right| > \left|\Lambda_{G}\right|.$
	\end{lem}
	\begin{proof}
		As $f_{i_1}$ and $f_{i_2}$ are linearly $F-$connected, there exist $f_{\ell_1},f_{\ell_2},\ldots,f_{\ell_m}$ in $F$ such that
		\begin{align}\label{eqn4point13}
			f_{i_1} = cf_{i_2} + \sum\limits_{1 \leq k \leq m}c_{\ell_k}f_{\ell_k},\;\;c,c_k \neq 0
		\end{align}
	and $\{f_{i_2}, f_{\ell_1},f_{\ell_2},\ldots,f_{\ell_m}\}$ is linearly independent. As, $f_{i_2} \notin span\{f_{\ell_1},f_{\ell_2},\ldots,f_{\ell_m}\},$ there exists $u\in \mathcal{H}_n$ such that $\langle u, f_{\ell_k} \rangle =0, $ for all $1 \leq k \leq m$ and $\langle u,f_{i_2} \rangle = \dfrac{1}{\bar{c_2}}\left(\dfrac{tr(K)}{N} - \langle g_{i_2},f_{i_2} \rangle  \right) \neq 0.$ Taking $U =\{u_i\}_{i=1}^N$ as
	\begin{align*}
		u_i = \begin{cases}
			-u, \;\;& for \;i=i_1,\\
			\bar{c_2}u,\;\; & for \;i= i_2, \\
			\bar{c_{i}},u \;\;& for \; i=\ell_k, 1 \leq k \leq m, \\
			0,\;\; &\text{otherwise},
		\end{cases}
	\end{align*}
     we get $G' = \{g_i + u_i\}_{i=1}^N$ as a dual of $F$ for,
\begin{align*}
	\sum\limits_{1\leq i \leq N} \langle f, u_i \rangle f_i =  \langle f, -u \rangle f_{i_1} +  \langle f, \bar{c_2}u \rangle f_{i_2} +\sum\limits_{1\leq k \leq m} \langle f, \overline{c_{\ell_k}}u \rangle f_{\ell_k}   = \langle f, u \rangle \left(\sum\limits_{1\leq k \leq m} c_{\ell_k} f_{\ell_k} + cf_{i_2} - f_{i_1} \right) =0,\; f \in \mathcal{H}_n,
\end{align*}
by \eqref{eqn4point13}. Now, $\langle g'_{i_2},f_{i_2} \rangle = \langle g_{i_2}+\bar{c_2}u,f_{i_2} \rangle = \dfrac{tr(K)}{N}$ and for $i= \ell_1,\ell_2, \ldots, \ell_m, \;  \langle g'_i,f_i \rangle = \langle g'_i + \bar{c_i}u,f_i \rangle = \langle g'_i,f_i \rangle + \bar{c_i}\langle u,f_i \rangle = \langle g'_i,f_i \rangle.$ Further, for $ i \in \{1,2,\ldots,N\}\setminus \{i_1,i_2,\ell_1,\ldots,\ell_m\},\;\langle g'_{i},f_{i} \rangle = \langle g_{i},f_{i} \rangle.$ This suggests that $\Lambda_{G}\subset \Lambda_{G'}$ and $i_2 \in \Lambda_{G'}.$ In other words, $|\Lambda_{G'}| > |\Lambda_{G}|.$
	\end{proof}

\begin{lem}\label{lemma4point10}
		Let $F = \{f_i\}_{i=1}^N $ be a $K-$frame for $\mathcal{H}_n,$ where $K$ be a positive semi-definite operator. Let $\mathcal{H}_n = H_1 \bigoplus H_2 \bigoplus \cdots \bigoplus H_i,$ where  ${H_i}'s$ are from linearly connected decomposition of $F.$ Let each $H_i,\,1\leq i \leq J$ is $K-$invariant. Then there exists a $K-$dual $G_j$ of $F_j$ such that $\langle g_i, f_i \rangle = \delta_{j}^K,$ for all $i \in \Lambda_j$ and for each $1 \leq j \leq J.$
		Moreover, there exist a $K-$dual $G'$ of $F$ such that $\langle g'_i , f_i \rangle = \delta_{j}^K,$ for all $i \in \Lambda_j.$ 
\end{lem}

\begin{proof}
	First we will prove that each $F_j = \{f_i\}_{i \in \Lambda_j}  $ is a linearly connected $K|_{H_j}$ frame for $H_j.$ For any $f \in H_j,$
	\begin{align*}
		\sum\limits_{i \in \Lambda_j}|\langle f,f_i \rangle|^2 &= \sum\limits_{i \in \Lambda_1}|\langle f,f_i \rangle|^2 + \sum\limits_{i \in \Lambda_2}|\langle f,f_i \rangle|^2 + \cdots + \sum\limits_{i \in \Lambda_J}|\langle f,f_i \rangle|^2 \\&= \sum\limits_{i =1}^N|\langle f,f_i \rangle|^2 \\&\geq A\|K^*f\|^2 \\&= A\|K^*|_{H_j}f\|^2,
	\end{align*}
for some constant $A>0,$ as $K^*$ is invariant on $H_j.$ Also, $	\sum\limits_{i \in \Lambda_j}|\langle f,f_i \rangle|^2 \leq \|f\|^2 \sum\limits_{i \in \Lambda_j}\|f_i\|^2 \leq B\|f\|^2,$ where $B=\sum\limits_{i=1}^N\|f_i\|^2>0.$ Therefore, $\{f_i\}_{i \in \Lambda_j}$ satisfy the $K|_{H_j}-$frame inequality for $H_j$ for all $1 \leq j \leq J.$ Hence, by Theorem\ref{theorem4point3}, there exist a $K|_{H_j}-$dual $G^j= \{g^j_{i}\}_{i\in \Lambda_j}$ such that $\langle g_i, f_i \rangle = \dfrac{tr(K|_{H_j})}{|\Lambda_j|}= \delta_{j}^K,$ for all $1 \leq j \leq J.$\\

Consider the collection $G= \{g_i\}_{i=1}^N =\bigcup\limits_{j=1}^J (P_{j})^* G^j,$ where $P_j$ is the projection of $\mathcal{H}_n$ onto $H_j.$ We shall show that $G$ is a dual of $F.$  Let $f \in \mathcal{H}_n$ and $x$ has the unique representation $f=f_1 + f_2 +\cdots + f_J,$ where $f_j = P_j(x),\; j=1,2,\ldots,J.$ Then,

\begin{align*}
	\sum_{i=1}^N \langle f,g_i \rangle f_i &=   \sum_{i \in \Lambda_1} \langle f,g_i \rangle f_i +  \sum_{i \in \Lambda_2} \langle f,g_i \rangle f_i + \cdots + \sum_{i \in \Lambda_J} \langle f,g_i \rangle f_i \\&=  \sum_{i \in \Lambda_1} \langle f,{(P_1)}^*g^1_i \rangle f_i +  \sum_{i \in \Lambda_2} \langle f,{(P_2)}^*g^2_i \rangle f_i + \cdots + \sum_{i \in \Lambda_J} \langle f,{(P_J)}^*g^J_i \rangle f_i  \\&=
	\sum_{i \in \Lambda_1} \langle f_1,g^1_i \rangle f_i +  \sum_{i \in \Lambda_2} \langle f_2,g^2_i \rangle f_i + \cdots + \sum_{i \in \Lambda_J} \langle f_J,g^J_i \rangle f_i \\&= f_1+f_2+\cdots +f_J = f,
\end{align*}
by utilizing the dual relationship between the sequences $\{f_i\}_{i \in \Lambda_j}$ and $\{g^j_i\}_{i \in \Lambda_j},$  $1 \leq j \leq J$. Furthermore, for $i \in \Lambda_j,\;1 \leq j \leq J,$ $\langle g_i,f_i \rangle =  \langle (P_j)^* g^j_i, f_i \rangle =  \langle g^j_i,f_i \rangle =  \dfrac{tr(K|_{H_j})}{|\Lambda_j|}= \delta_{j}^K.$ \\

\end{proof}

	\begin{thm}
Let $F = \{f_i\}_{i=1}^N$ be a $K$-frame for $\mathcal{H}_n$, where $K$ is a positive semi-definite operator and $\{\Lambda_j\}$, $\{H_j\}$ denote the index sets and corresponding subspaces arising from a linearly connected decomposition of $F$. Assume that each $H_j$ is invariant under $K$. If $\{\delta_{j}^K\}_{j=1}^J$ denotes the $K-$redundancy distribution of $F$, then $\tilde{r}_{1}(F) = \max\limits_{1 \leq j \leq J} \delta_{j}^K = \max\limits_{1 \leq j \leq J} \dfrac{\operatorname{tr}(K|_{H_j})}{|\Lambda_j|}.$
\end{thm}

	\begin{proof}
		   For any $K-$dual $G=\{g_i\}_{i=1}^N$ of $F,\,$ $\left\{P_j g_i\right\}_{i \in \Lambda_j}$ is a $K|_{H_j}-$ dual of $F_j=\{f_i\}_{i \in \Lambda_j},$ for all $1 \leq j \leq J.$ For any $h \in H_j,$
		  \begin{align*}
		  	K|_{H_j} (h) &= K(h) \\&= \sum\limits_{i=1}^N \langle h, g_i \rangle f_i \\&= \sum\limits_{i=1}^N \langle P_j h, g_i \rangle f_i \\& = \sum\limits_{i \in \Lambda_1} \langle  h, P_1 g_i \rangle f_i + \sum\limits_{i \in \Lambda_2} \langle  h, P_2 g_i \rangle f_i + \cdots +\sum\limits_{i \in \Lambda_J} \langle  h, P_J g_i \rangle f_i \\&= \sum\limits_{i \in \Lambda_j} \langle  h, P_j g_i \rangle f_i  .
		  \end{align*}
	  It is easy to see that $\langle P_j g_i,f_i \rangle = \langle  g_i, P_j f_i \rangle  =\langle  g_i,f_i \rangle ,$ for all $ i \in \Lambda_j$ and $1\leq j \leq J.$ Thus,
	  $ r_1 \left(F_j, \left\{P_j g_i\right\}_{i \in \Lambda_j} \right) = \max\limits_{i \in \Lambda_j} \left| \langle f_i, p_j g_i \rangle   \right| \geq  \dfrac{tr(K|_{H_j})}{|\Lambda_j|}= \delta_{j}^K$ and hence, $\max\limits_{i \in \Lambda_j} \left|  \langle g_i, f_i \rangle \right| \geq \delta_{j}^K,\,$ for all $1 \leq j \leq J.$ Consequently, $r_{1} (F,G) = \max\limits_{1 \leq i \leq N} \left|  \langle g_i, f_i \rangle \right| = \max\limits_{1 \leq j \leq J} \max\limits_{i \in \Lambda_j} \left|  \langle g_i, f_i \rangle \right|\geq \max\limits_{1 \leq j \leq J}\delta_{j}^K.$  By Lemma\ref{lemma4point10}, $F$ has a $K-$dual $G'$ such that $\langle g'_i, f_i \rangle = \delta_{j}^K,$ for all $i \in \Lambda_j$ and for each $1 \leq j \leq J.$  Thus  $r_{1} (F,G') =\max\limits_{1 \leq j \leq J} \delta_{j}^K$ and hence, $\tilde{r}_{1} (F) = \max\limits_{1 \leq j \leq J} \delta_{j}^K.$

	\end{proof}

\noindent	The following theorem establishes a sufficient condition ensuring that the canonical dual serves as a $1-$erasure spectrally optimal dual for a given Parseval \( K-\)frame \( F = \{f_i\}_{i=1}^N \). Define $M = \max \left\{ \left\| K^{\dagger\,1/2} f_i \right\|^2 : 1 \leq i \leq N \right\},\\ J_1 = \left\{ i : \left\| K^{\dagger\,1/2} f_i \right\|^2 = M \right\}, \; J_2 = \{1,2,\dots,N\} \setminus J_1$ and $W_j = \operatorname{span} \left\{ f_i : i \in J_j \right\}.$

	\vskip 1em
	\begin{thm}\label{prop3point4}
		Let $F = \{f_i\}_{i=1}^N $ be a Parseval $K-$frame for $\mathcal{H}_n,$ where $K$ be a positive semi-definite operator on $\mathcal{H}_n.$ If $W_1 \cap W_2 = \{0\}$, then the canonical $K-$dual is indeed a $1-$erasure spectrally optimal $K-$dual of $F$.

	\end{thm}
	
	\begin{proof}
		Let  $G = \{g_i\}_{i=1}^N = \{K^\dag f_i + u_i \}_{i=1}^N$ be a dual of $F,$   wherein $\displaystyle{ \sum_{i=1}^N \langle f,u_i \rangle f_i = 0} ,\; \text{for all }f \in \mathcal{H}_n.$ Further, using the condition $W_1 \cap W_2 = \{0\},$ we obtain $\sum\limits_{i \in  J_1 } \langle f,u_i \rangle f_i =0, \;f \in \mathcal{H}_n.$ This leads to $\Theta_{F_1}^* \Theta_{U_1} = 0, \; \textit{where} \; F_1 = \{f_i\}_{i \in  J_1}\; \textit{and} \;\; U_1 = \{u_i\}_{ i \in  J_1},$ and so	$	tr\left(\Theta_{F_1}^* \Theta_{U_1}\right) = tr\left( \Theta_{U_1} \Theta_{F_1}^* \right)  =\sum\limits_{i \in  J_1 } \langle f_i , u_i \rangle = 0.$ Consequently,
		\begin{align}\label{equation4point14*}
			Re \left(\sum\limits_{i \in  J_1 } \langle f_i , u_i \rangle \right) = 0.
		\end{align}
		It is easy to see that
		\begin{eqnarray*}
			\max\limits_{1 \leq i \leq N} \;|\langle f_i , g_i \rangle |\; \geq \;  \max_{i \in J_1} \; |\langle f_i, K^\dag f_i  \rangle + \langle f_i,u_i \rangle| =  \max_{i \in J_1} \left|  M +  \langle f_i,u_i\rangle\right|.
		\end{eqnarray*}
		If  $Re  \langle f_i , u_i \rangle =0,\,\forall i \in J_1,$ then $\max\limits_{1 \leq i \leq N} \; |\langle f_i , g_i \rangle | \geq M.$ Now, assume that there is an element \( j \in J_1 \) satisfying $Re  \langle f_j , u_j \rangle > 0,$ then $\max\limits_{1 \leq i \leq N} |\langle f_i , g_i \rangle | >M.$ We observe that if $Re \langle f_j,u_j \rangle <0$ for some $j \in J_1,$ then there exists a $j' \in J_1$ for which $Re \langle f_{j'}, u_{j'}\rangle >0,$ by \eqref{equation4point14*}. Therefore, $r_{1} (F,G) \geq M = r_{1} (F, K^\dag F) .$ Consequently, the canonical dual emerges as a $1-$erasure spectrally optimal $K-$dual corresponding to $F$.

	\end{proof}
   The theorem below states a sufficient condition that ensures the canonical dual of a Parseval $K-$frame is $1-$erasure spectrally optimal $K-$dual. It also establishes the non-uniqueness of such optimal duals by showing that their number is actually uncountable.
	
		\begin{thm}\label{thm3point4}
		Let $F = \{f_i\}_{i=1}^N $ be a Parseval $K-$frame for $\mathcal{H}_n$ with $N>n$ and $K$ be a positive semi-definite operator on $\mathcal{H}_n.$ Let $M, J_1, J_2, W_1, W_2$ be as in Theorem \ref{prop3point4}. If $W_1 \cap W_2 = \{0\}$ and $\{f_i\}_{i \in J_1}$ is linearly independent then the canonical dual is a $1-$erasure spectrally optimal $K-$dual of $F.$ Moreover, the number of  $1-$erasure spectrally optimal $K-$dual of $F$ is uncountable.
	\end{thm}
	\begin{proof}
		By Theorem\ref{prop3point4}, the canonical dual $K^\dag F$ is a $1-$erasure spectrally optimal $K-$dual of $F.$ As $N>n,$ we can find a $K-$dual  $G=  \{g_i\}_{i=1}^N = \{K^\dag f_i + u_i \}_{i=1}^N $ of $F$ with $u_i = 0$ for all $i \in J_1$ and $u_i \neq 0$ for some $i \in J_2.$ It is easy to see that for any $t \in \mathbb{R},$ $G_t =  \{K^\dag f_i + tu_i \}_{i=1}^N $ is a $K-$dual of $F$ as $\sum\limits_{i=1}^N \langle f, tu_i\rangle f_i = \sum\limits_{i \in J_1} \langle f, tu_i\rangle f_i + \sum\limits_{i \in J_2} \langle f, tu_i\rangle f_i =  t\sum\limits_{i \in J_2} \langle f, u_i\rangle f_i =0. $ For $i \in J_1,\,$ $\left| \langle K^\dag f_i + u_i , f_i \rangle \right| = M.$ For each $i \in J_2,\,$  $\left| \langle K^\dag f_i , f_i \rangle \right| < M.$ As $t \to \left| \langle K^\dag f_i +tu_i , f_i \rangle \right|$ is a continuous function on $\mathbb{R},$ then  there exists $\alpha_i > 0$ such that $ \left| \langle K^\dag f_i +tu_i , f_i \rangle \right| < M,$ when $t \in (-\alpha_i, \alpha_i),$   for each $i \in J_2.$ Taking $\alpha = \min\limits_{i \in J_2} \alpha_i,$ we can say  for each $i \in J_2,\,$ $\left| \langle K^\dag f_i + tu_i , f_i \rangle \right| < M,\,\forall t \in (-\alpha, \alpha).$ Therefore, for any $t \in  (-\alpha, \alpha),\,$ $r_{1}(F,G_t) = M = r_{1}(F,K^\dag F)$ and hence $G_t$ is a $1-$erasure  spectrally optimal $K-$dual of $F.$
	\end{proof}

	\begin{cor}
Let $F = \{f_i\}_{i=1}^N$ be a Parseval $K-$frame for $\mathcal{H}_n$. Suppose  that $\|f_i\|\,\|K^\dag f_i\|$ is a constant for all $1 \leq i \leq N$. Then the canonical $K-$dual $K^\dag F$ of $F$ is the unique $1-$erasure optimal $K-$dual of $F$ with respect to the operator norm.
\end{cor}

	The following proposition presents a necessary condition for the spectral optimality of the canonical $K-$dual. It shows that under certain linear dependence among frame elements, the canonical $K-$dual fails to be $1-$erasure spectrally optimal.
	
	\begin{prop}
Let $F = \{f_i\}_{i=1}^N$ be a Parseval $K-$frame for $\mathcal{H}_n$, where $K \in \mathcal{B}(\mathcal{H}_n)$ is a positive semi-definite operator. Assume that the subset $\{f_i\}_{i \in J_1}$ is linearly independent, and there exists a sequence $\{d_i\}_{i=1}^N$ with $d_i \neq 0$ for all $i \in \Lambda_1$ such that $\sum_{i=1}^N d_i f_i = 0.$ Then the canonical $K-$dual $K^\dag F$ of $F$ is not a $1-$erasure spectrally optimal $K-$dual of $F$.
\end{prop}

	\begin{proof}
		Since $\{f_i\}_{i \in \Lambda_1}$ is linearly independent, then $\{d_i f_i\}_{i \in \Lambda_1}$ is also linearly independent and hence there exists a $h' \in \mathcal{H}_n$ such that for all $i \in J_1,$ $\langle f_i, \bar{d}_i h \rangle = \langle d_i K^\dag f_i, h \rangle < 0,$ by Lemma\ref{lem3point6}. Let $V=\{v_i\}_{i=1}^N,$ where $v_i = \bar{d}_i h;1 \leq i \leq N. $ It is easy to see that $\sum\limits_{i=1}^N \langle f, f_i \rangle v_i = \left \langle f, \sum\limits_{i=1}^N d_i f_i \right \rangle h = 0.$ Therefore, $\{K^\dag f_i + tv_i\}_{i=1}^N$ is a $K-$dual of $F$ for all $t.$ Now, for all $i \in \Lambda_1,$
		\begin{align*}
			\langle K^\dag f_i +tv_i, f_i \rangle  = \left\|\left(K^\dag \right)^{\frac{1}{2}} f_i \right\|^2 +  t\langle v_i, f_i \rangle= \left|	\langle K^\dag f_i, f_i \rangle \right| +t \langle u_i, f_i \rangle = M +  t \langle u_i, f_i \rangle
		\end{align*} 
		As $t \to \left| M+ t \langle u_i, f_i \rangle \right|$ is a  continuous  function on $\mathbb{R}$ then there exist $t_i >0 $ such that  $\left| M+ t \langle u_i, f_i \rangle \right| < M$ for all $t \in (-t_i, t_i)$ and for each $i \in J_1.$ As for each $j \in J_2,\;$ $\langle K^\dag f_i , f_i \rangle < M,$ then there exist $t'_j>0$ small enough such that, for any $t \in (-t'_j, t'_j),\;$  $\left| \langle K^\dag f_j +tv_j, f_j \rangle \right| = \left| \langle K^\dag f_j , f_j \rangle + t \langle v_j, f_j \rangle \right|  <M,$ for each $j \in \Lambda_2.$ Taking $\widetilde{t} = \min\limits_{i \in \Lambda_1, j \in J_2} \left\{t_i, t'_j\right\},$ we have, for $t \in (-\tilde{t}, \tilde{t} ),\;$ $\left| \langle K^\dag f_i +tv_i, f_i \rangle \right| < M,$ for any $1 \leq i \leq N.$ Hence, for the dual $G_t = \{K^\dag f_i + tv_i\}_{i=1}^N ,$\; $r_{1}(F,G_t) <M = r_{1}(F,G),\;\forall t \in  (-\tilde{t}, \tilde{t} ).$
		
	\end{proof}

	Next, under suitable conditions, we derive a simplified expression for $r_{2}(F,G)$ in terms of $r_{1}(F,G)$, as stated in the following two theorems.

\begin{prop} \label{thm3point2}
		Let $F= \{f_i\}_{i=1}^N$ be a $K-$frame for $\mathcal{H}_n.$  Let $G= \{g_i\}_{i=1}^N$ be a dual of $F$ satisfying
		\begin{enumerate}
			\item [{\em (i)}] $\langle g_i , f_i \rangle \geq 0,  \; \textit{for all}\; 1\leq i \leq N$
			\item [{\em (ii)}] $ \langle g_i , f_j \rangle \langle g_j , f_i \rangle  = c \geq 0 , \; \text{for all}\;\; i \neq j.$
		\end{enumerate}
		Then,
		
		\begin{eqnarray*}
			r_{2} (F,G) = \begin{cases}\frac{1}{2} \left( r_{1} (F,G) + \max\limits_{i \in \Delta^c } \langle g_i , f_i \rangle + \sqrt{\left( r_{1} (F,G)  - \max\limits_{i \in \Delta^c } \langle g_i , f_i \rangle \right)^2 + 4c} \right),\;\; & \text{if}\; c>0 \;\text{and}\; |\Delta|=1,\\~\\
				r_{1} (F,G) + \sqrt{c}, \;\;& \text{if}\; c>0 \;\text{and}\; |\Delta|>1,\\~\\
				r_{1} (F,G),  \;\;& \text{if}\; c=0,
				
			\end{cases}
		\end{eqnarray*}
		where $\Delta =  \bigg\{ i: \langle g_i , f_i \rangle  = r_{1} (F,G)\bigg\}. $

	\end{prop}
	
	\begin{proof}
		Suppose that $c > 0.$ We begin by considering the case when $\Delta = \{k\}.$ Let $\Delta_1 := \left\{i : \max\limits_{j \in \Delta^c} \, \langle g_j , f_j \rangle  = \langle g_i , f_i \rangle  \right\}$ and \, $ \ell \in \Delta_1.$ A straightforward computation shows that for any $i$ the function  $\theta_i(x) = \frac{1}{2} \left| x +\alpha_{ii} + \sqrt{(x -\alpha_{ii})^{2}+4c} \right |,$ where $\alpha_{ij} = \langle g_i,f_j \rangle,$ is an increasing function on $[0,\infty).$ This would result in $\max\limits_{\substack{1 \leq j \leq N \\ j \neq i}} \theta_i(\alpha_{jj}) = \theta_i( \alpha_{kk}),$ when $i \neq k$ and $\max\limits_{\substack{1 \leq j \leq N \\ j \neq k}} \theta_k( \alpha_{jj})= \theta_k( \alpha_{\ell \ell}). $ Therefore,
		\begin{align*}
			r_{2}(F,G) = \max\limits_{1 \leq i \leq N} \max\limits_{\substack{1 \leq j \leq N \\ j \neq i}} \theta_i( \alpha_{jj})
			= \max \bigg\{ \theta_k( \alpha_{\ell \ell}), \theta_i( \alpha_{kk}) : 1 \leq i \leq N, i \neq k \bigg\}.
		\end{align*}
		As $\theta_i( \alpha_{jj}) = \theta_j( \alpha_{ii}),$ for $i \neq j,$ we have
		\begin{align}\label{equation4point15}
			r_{2} (F,G) &= \max\limits_{i \neq k } \theta_k( \alpha_{ii}) = \theta_k( \alpha_{\ell \ell}) = \frac{1}{2}\left|   \alpha_{\ell \ell}  +  \alpha_{kk} +  \sqrt{( \alpha_{\ell \ell} - \alpha_{kk})^{2}+4c}   \right| \\&= \frac{1}{2} \left( r_{1} (F,G) + \max\limits_{i \in \Delta^c } \langle g_i , f_i \rangle + \sqrt{\left( r_{1} (F,G)  - \max\limits_{i \in \Delta^c } \langle g_i , f_i \rangle \right)^2 + 4c} \right). \nonumber
		\end{align}
		
		On the other hand, suppose $|\Delta| \geq 2.$ Let $k_1,k_2 \in \Delta.$ Then, proceeding along the lines of the earlier argument, we get $k_1$ and $k_2$ in the place of $k$ and $\ell$ in \eqref{equation4point15}. Hence, $ r_{2} (F,G) = r_{1} (F,G) + \sqrt{c}.$\\
		Let us now assume that \( c = 0 \). Utilizing part (i) and the representation of \( r_{2}(F, G) \) as provided in \eqref{eqn3point3}, it follows that

\[
r_{2}(F, G) = \frac{1}{2} \max\limits_{i \neq j} \left\{ \alpha_{ii} + \alpha_{jj} + \sqrt{(\alpha_{ii} - \alpha_{jj})^2} \right\} = \max\limits_{1 \leq i \leq N} \alpha_{ii} = r_{1}(F, G),
\]
thereby completing the proof of the theorem.

	\end{proof}

	\begin{thm} \label{thm3point2*}
		Let $F= \{f_i\}_{i=1}^N$ be a $K-$frame for $\mathcal{H}_n$. Let $G= \{g_i\}_{i=1}^N$ be a $K-$dual of $F$ satisfying
		\begin{enumerate}
			\item [{\em (i)}] $\langle g_i , f_i \rangle \geq 0,  \; \textit{for all}\; 1\leq i \leq N.$
			\item [{\em (ii)}] $ \langle g_i , f_j \rangle \langle g_j , f_i \rangle  = c < 0 , \; \text{for all}\;\; i \neq j.$
		\end{enumerate}
		If  $\Delta =  \bigg\{ i: \langle g_i , f_i \rangle  = r_{1} (F,G)\bigg\} $ contains more than one element, then \; $r_{2} (F,G) = \sqrt{\left(r_{1} (F,G)\right)^2-c}.$
	\end{thm}

	\begin{proof}
		By using condition (ii) in \eqref{eqn3point3}, we obtain
		\begin{align*}
			r_{2} (F,G) &= \frac{1}{2} \max_{i\neq j} \left|  \alpha_{ii} + \alpha_{jj} \pm \sqrt{( \alpha_{ii} -  \alpha_{jj})^2 + 4c}\right| .
		\end{align*}
		One can verify that  $r_{2} (F,G)  =  \frac{1}{2} \max\limits_{i\neq j} \left|  \alpha_{ii} + \alpha_{jj} + \sqrt{( \alpha_{ii} -  \alpha_{jj})^2 + 4c}\right|, $ whether $\sqrt{( \alpha_{ii} -  \alpha_{jj})^2 + 4c }$ is real or complex. Define $\gamma_{ij} := \frac{1}{2} \left|  \alpha_{ii} + \alpha_{jj} + \sqrt{( \alpha_{ii} -  \alpha_{jj})^2 + 4c}\right|.$ It is an easy computation that
		\begin{align*}
			\gamma_{ij} = \begin{cases}
				\frac{1}{2} \left( \alpha_{ii} + \alpha_{jj} + \sqrt{( \alpha_{ii} -  \alpha_{jj})^2 + 4c}\right),\;\; & \text{if}\; ( \alpha_{ii} -  \alpha_{jj})^2 + 4c \geq 0,\\~\\
				\sqrt{ \alpha_{ii}\alpha_{jj}-c}, \;\;& \text{if}\; ( \alpha_{ii} -  \alpha_{jj})^2 + 4c < 0.
			\end{cases}
		\end{align*}
		Let $k_1,k_2 \in \Delta.$ Then clearly, $\gamma_{k_1 k_2} = \sqrt{\left(r_{1} (F,G)\right)^2-c}.$ Let $\ell_1 \notin \Delta.$ When $\left(  r_{1} (F,G) -  \alpha_{\ell_1 \ell_1}\right)^2 +4c \geq 0,$ we obtain
		\begin{align*}
			\gamma_{k_1 \ell_1 } &= \frac{1}{2}  \left( r_{1} (F,G)  +  \alpha_{\ell_1 \ell_1} + \sqrt{\left(  r_{1} (F,G) -  \alpha_{\ell_1 \ell_1}\right)^2 +4c}\; \right) \\&< \frac{1}{2}  \left( r_{1} (F,G)   + \alpha_{\ell_1 \ell_1} + \sqrt{\left(  r_{1} ( F,G) -   \alpha_{\ell_1 \ell_1}\right)^2 }\right) \\&=  r_{1} (F,G),
		\end{align*}
		as $ \alpha_{\ell_1 \ell_1} < r_{1} (F,G).$ Moreover, $r_{1} (F,G) < \sqrt{\left(r_{1} (F,G)\right)^2-c}. $ So, $\gamma_{k_1 \ell_1 } < \sqrt{\left(r_{1} (F,G)\right)^2-c}.$ On the other hand, when $\left(  r_{1} (F,G) -   \alpha_{\ell_1 \ell_1}\right)^2 +4c < 0,$ then clearly
		$$\gamma_{k_1 \ell_1 } =  \sqrt{r_{1} (F,G)  \alpha_{\ell_1 \ell_1} -c} < \sqrt{\left(r_{1} (F,G)\right)^2-c}.$$
		Now, let us take $\ell_2 \notin \Delta.$ All that remains to prove is that $\gamma_{ \ell_1 \ell_2 } \leq \sqrt{\left(r_{1} (F,G)\right)^2-c}. $  First, we consider the case when $\left(\alpha_{\ell_1 \ell_1} -  \alpha_{\ell_2 \ell_2} \right)^2 + 4c \geq 0.$  Suppose, on the contrary, we have $\sqrt{\left(r_{1} (F,G)\right)^2-c} < \gamma_{ \ell_1 \ell_2 }= \frac{1}{2} \left( \alpha_{\ell_1 \ell_1}   +  \alpha_{\ell_2 \ell_2} + \sqrt{\left( \alpha_{\ell_1 \ell_1} -   \alpha_{\ell_2 \ell_2}\right)^2 +4c}\right) .$ Without loss of generality, we may assume that $ \alpha_{\ell_2 \ell_2} \leq  \alpha_{\ell_1 \ell_1}.$ Moreover, as  $ \alpha_{\ell_1 \ell_1} <r_{1} (F,G),$ we obtain $0 \leq \left( \alpha_{\ell_1 \ell_1} - \alpha_{\ell_2 \ell_2}\right)^2 + 4c < \left(r_{1} (F,G) - \alpha_{\ell_2 \ell_2} \right)^2 + 4c. $ Therefore, we get  $\sqrt{\left(r_{1} (F,G)\right)^2-c} < \frac{1}{2} \left( r_{1} (F,G) + \alpha_{\ell_2 \ell_2} + \sqrt{\left( r_{1} (F,G)  -  \alpha_{\ell_2 \ell_2} \right)^2 + 4c} \right).$ Then,
		\begin{align}\label{eqn4point16}
			\sqrt{4\left(r_{1} (F,G)\right)^2-4c} - \sqrt{\left( r_{1} (F,G) -  \alpha_{\ell_2 \ell_2}\right)^2 +4c} < r_{1} (F,G) +   \alpha_{\ell_2 \ell_2} .
		\end{align}
		Now, $\left( r_{1} (F,G) -  \alpha_{\ell_2 \ell_2}\right)^2  < 4\left(r_{1} (F,G)\right)^2, $  which implies that $\sqrt{4\left(r_{1} (F,G)\right)^2-4c}\; - \\ \sqrt{\left( r_{1} (F,G) -   \alpha_{\ell_2 \ell_2}\right)^2 +4c} > 0.$ Therefore, squaring \eqref{eqn4point16} and simplifying, we obtain
		$$ \left(r_{1} (F,G)\right)^2 - r_{1} (F,G)   \alpha_{\ell_2 \ell_2} < \sqrt{\left(r_{1} (F,G)\right)^2-c} \; \sqrt{\left( r_{1} (F,G) -   \alpha_{\ell_2 \ell_2}\right)^2 +4c}. $$
		The left hand side is positive and hence on squaring again, we get
		\begin{align*}
			\left(r_{1} (F,G)\right)^2 &\left( r_{1} (F,G) -  \alpha_{\ell_2 \ell_2}    \right) ^2  \\&	<\left(r_{1} (F,G)\right)^2 \left( r_{1} (F,G) -  \alpha_{\ell_2 \ell_2}    \right) ^2  + 4c \left(r_{1} (F,G)\right)^2 -c\left(  r_{1} (F,G) -  \alpha_{\ell_2 \ell_2}     \right)^2 -4c^2,
		\end{align*}
		which may be simplified as
		$$ 0< c\left(\left(2{r_{1}} (F,G)\right)^2 - \left( r_{1} (F,G) -  \alpha_{\ell_2 \ell_2}    \right) ^2 -4c  \right)  = c\bigg(   \left(3r_{1} (F,G) -\alpha_{\ell_2 \ell_2}    \right)    \left(r_{1} (F,G) + \alpha_{\ell_2 \ell_2}  \right) -4c \bigg).$$
		We observe that each of the terms \( \left(3r_{1}(F,G) - \alpha_{\ell_2 \ell_2} \right) \), \( \left(r_{1}(F,G) + \alpha_{\ell_2 \ell_2} \right) \), and the term \( -4c \) on the right-hand side is strictly positive. Consequently, the above inequality is not valid. Therefore,  $\gamma_{ \ell_1 \ell_2 } \leq \sqrt{\left(r_{1} (F,G)\right)^2-c}.$\\
		Finally, if  $\left(  \alpha_{\ell_1 \ell_1} - \alpha_{\ell_2 \ell_2} \right)^2 + 4c < 0,$ then $ \gamma_{\ell_1 \ell_2} =\sqrt{ \alpha_{\ell_1 \ell_1} \alpha_{\ell_2 \ell_2} -c}< \sqrt{\left(r_{1} (F,G)\right)^2-c}.$ Thus, we conclude that $	r_{2} (F,G) =\max\limits_{i \neq j} \gamma_{ij} =  \sqrt{\left(r_{1} (F,G)\right)^2-c}.$

	\end{proof}

	We now present a sufficient condition under which a given $K-$dual frame achieves spectral optimality with respect to both one and two erasures.

\begin{thm}
Let \( F \) be a \( K-\)frame for a real Hilbert space \( \mathcal{H}_n \). If \( G \) is a \( 2-\)uniform \( K-\)dual of \( F \), then \( G \) is spectrally optimal for two erasures.
\end{thm}

	\begin{proof}
			First consider the case when $tr(K)\geq 0.$ For a $2-$uniform $K-$dual frame $G$ of $F,$ we have, $ \langle f_i,g_i \rangle = \dfrac{tr(K)}{N},\; \forall i$  and so $r_{1}(F,G) =\dfrac{tr(K)}{N}.$  As $r_{1}(F,G') \geq \dfrac{tr(K)}{N},$ for any $K-$dual $G'$ of $F,$ as established in Theorem~\ref{thm5point2}, we have that $\tilde{r}_{1} (F) = \dfrac{tr(K)}{N}$ and further, $G$ is a $1-$erasure spectrally optimal $K-$dual of $F.$ The fact that  $ \langle f_i,g_i \rangle = \dfrac{tr(K)}{N},\; \forall\, i$ also suggests that $\langle f_i,g_i \rangle \geq 0\; \forall\;i$ and $|\Delta| = N,$ where $\Delta$ is as in Theorem \ref{thm3point2}. Therefore, by Theorems \ref{thm3point2} and \ref{thm3point2*},
		$$ r_{2}(F,G) = \begin{cases}
			\dfrac{tr(K)}{N} + \sqrt{c},\;\;  & \text{if $c \geq 0,$}\\~\\
			\sqrt{\dfrac{tr(K)}{N} -c},\;\;\;\; & \text{if $c < 0$},
		\end{cases}$$
		where $c=   \langle g_i,f_j \rangle \langle g_j,f_i \rangle,\;i \neq j.$ We may write
		$$cN(N-1) = \sum\limits_{i \neq j}  \langle g_i,f_j \rangle \langle g_j,f_i \rangle = tr(K) - \dfrac{\left(tr(K)\right)^2}{N}, $$
		using \eqref{eqation7}, from which we obtain
		\begin{align}\label{equation4point16}
			c = \dfrac{tr(K) - \dfrac{\left(tr(K)\right)^2}{N}}{N(N-1)}.
		\end{align}
Explicitly, $$r_{2} (F,G) = \begin{cases}
	\dfrac{tr(K)}{N} + \sqrt{\dfrac{tr(K) - \dfrac{\left(tr(K)\right)^2}{N}}{N(N-1)}},\;\;  & \text{if $tr(K) \geq \dfrac{\left(tr(K)\right)^2}{N},$}\\~\\
	\sqrt{\left(\dfrac{tr(K)}{N}\right)^2 -\dfrac{tr(K) - \dfrac{\left(tr(K)\right)^2}{N}}{N(N-1)}},\;\;\;\; & \text{if $tr(K) < \dfrac{\left(tr(K)\right)^2}{N}$}.
\end{cases}$$	
	
		Next, we shall show that $r_{2} (F,G') \geq \begin{cases}
			\dfrac{tr(K)}{N} + \sqrt{c},\;\;  & \text{if $c \geq 0,$}\\~\\
			\sqrt{\left(\dfrac{tr(K)}{N}\right)^2 -c},\;\;\;\; & \text{if $c<0$},
		\end{cases} $
		\;\; for any $G' \in \mathcal{R}_{1}(F),$ so that  $\tilde{r}_{2} (F) =  r_{2} (F,G).$ For $G' \in \mathcal{R}_{1}(F),\; r_{1} (F,G') = \tilde{r}_{1} (F) = \dfrac{tr(K)}{N}.$ In other words, $\max\limits_{1 \leq i \leq N}  \left| \langle f_i, g'_i \rangle  \right| = \dfrac{tr(K)}{N},$ which implies $  \langle f_i, g'_i \rangle \leq \dfrac{tr(K)}{N},\;\forall\,i,$ as $\mathcal{H}_n$ is a real Hilbert space. We note that $ tr(K) = \sum\limits_{1 \leq i \leq N} \langle f_i, g'_i \rangle \leq \sum\limits_{1 \leq i \leq N} \dfrac{tr(K)}{N} = tr(K),$ which forces  $ \langle f_i, g'_i \rangle = \dfrac{tr(K)}{N}, \, \forall\,i.$ Equivalently, $G'$ is a $1-$uniform $K-$dual of $F.$ Therefore,
		$$r_{2} (F,G') = \max_{i \neq j} \left| \dfrac{tr(K)}{N}+ \sqrt{ \langle g'_i, f_j \rangle \langle g'_j, f_i \rangle}     \right|,  $$
		analogous to \eqref{eqn3point4*}.
		
		Let us now analyze the scenario in which \( c \geq 0 \). There exist $i_0 \neq j_0$ such that $ \langle g'_{i_0}, f_{j_0} \rangle \langle g'_{j_0}, f_{i_0} \rangle \geq \dfrac{\sum\limits_{i \neq j} \langle g'_i, f_j \rangle  \langle g'_j, f_i \rangle}{N(N-1)} = \dfrac{tr(K) - \frac{\left(tr(K)\right)^2}{N}}{N(N-1)},$ as demonstrated in the proof of Theorem~\ref{theorem3point5twoerasure}. By employing the \(1-\)uniformity property of $G',$ \eqref{eqation7} and   \eqref{equation4point16}, we get  $ \langle g'_{i_0}, f_{j_0} \rangle \langle g'_{j_0}, f_{i_0} \rangle \geq c,$ which in turn leads to
		$$r_{2} (F,G')  \geq \left|\frac{tr(K)}{N} + \sqrt{ \langle g'_{i_0}, f_{j_0} \rangle \langle g'_{j_0}, f_{i_0} \rangle}   \right| \geq \frac{tr(K)}{N} + \sqrt{c}. $$
		On the other hand, when $c<0$ then $tr(K) < \dfrac{\left(tr(K)\right)^2}{N}.$ Therefore, $ \sum\limits_{i \neq j}  \langle g'_i,f_j \rangle \langle g'_j,f_i \rangle = tr(K) - \dfrac{\left(tr(K)\right)^2}{N}<0$ gives $\min\limits_{i \neq j} \langle g'_i,f_j \rangle \langle g'_j,f_i \rangle \leq \dfrac{tr(K) - \dfrac{\left(tr(K)\right)^2}{N}}{N(N-1)}.$ There exist $i_1 \neq j_1$ such that $ \langle g'_{i_1}, f_{j_1} \rangle \langle g'_{j_1}, f_{i_1} \rangle \leq \dfrac{\sum\limits_{i \neq j} \langle g'_i, f_j \rangle  \langle g'_j, f_i \rangle}{N(N-1)} = \dfrac{tr(K) - \dfrac{\left(tr(K)\right)^2}{N}}{N(N-1)} = c <0.$ This leads to  
		$$r_{2} (F,G')  \geq \left|\frac{tr(K)}{N} + \sqrt{ \langle g'_{i_1}, f_{j_1} \rangle \langle g'_{j_1}, f_{i_1} \rangle}   \right| = \sqrt{\left( \frac{tr(K)}{N}\right)^2 - \langle g'_{i_1}, f_{j_1} \rangle \langle g'_{j_1}, f_{i_1} \rangle} \geq  \sqrt{\left( \frac{tr(K)}{N}\right)^2 - c}. $$

	Hence, \( G \) is a $2-$erasure spectrally optimal $K-$dual of \( F \). Employing the expression for \( c \) from \eqref{equation4point16}, we may express

		$$ \tilde{r}_{2} (F) = r_{2} (F,G) =  \begin{cases}
			\dfrac{tr(K)}{N} + \sqrt{\dfrac{tr(K) - \dfrac{\left(tr(K)\right)^2}{N}}{N(N-1)}},\;\;  & \text{if $tr(K) \geq \dfrac{\left(tr(K)\right)^2}{N},$}\\~\\
			\sqrt{\left(\dfrac{tr(K)}{N}\right)^2 -\dfrac{tr(K) - \dfrac{\left(tr(K)\right)^2}{N}}{N(N-1)}},\;\;\;\; & \text{if $tr(K) < \dfrac{\left(tr(K)\right)^2}{N}$}.
		\end{cases} .$$
	
	\par Now consider the case when $tr(K)< 0.$ For a $2-$uniform $K-$dual frame $G$ of $F,$ we have, $ \langle f_i,g_i \rangle = \dfrac{tr(K)}{N},\; \forall i$  and so $r_{1}(F,G) =\max\limits_{1\leq i \leq N} \left|\langle f_i, g''_i \rangle \right| = -\dfrac{tr(K)}{N}.$ For any $K-$dual $G''$ of $F,\,$  $r_{1}(F,G'') \geq -\dfrac{tr(K)}{N}.$ If not, then $\left|\langle f_i, g''_i \rangle \right| < -\dfrac{tr(K)}{N},\;1\leq i \leq N.$ Therefore, $-tr(K) = \left|\sum\limits_{1\leq i \leq N} \langle f_i, g''_i \rangle \right| \leq \sum\limits_{1\leq i \leq N} \left|\langle f_i, g''_i \rangle \right| < -tr(K),$ which is not possible. Thus, $G$ is a $1-$erasure spectrally optimal $K-$dual of $F.$ As $G$ is a $2-$uniform $K-$dual of $F,$ we then have $\langle f_i, g_j \rangle \langle f_j, g_i \rangle =\dfrac{tr(K) - \frac{\left(tr(K)\right)^2}{N}}{N(N-1)}< 0,$ for all $i \neq j.$ By \eqref{eqn3point4*}, 
	$$r_{2}(F,G) = \sqrt{\left(\dfrac{tr(K)}{N} \right)^2 - \dfrac{tr(K) - \frac{\left(tr(K)\right)^2}{N}}{N(N-1)}}.$$
	Now we shall show that for any $K-$dual $G''' \in \mathcal{R}_1 (F),\;$   $r_{2}(F,G''') \geq \sqrt{\left(\dfrac{tr(K)}{N} \right)^2 - \dfrac{tr(K) - \frac{\left(tr(K)\right)^2}{N}}{N(N-1)}}.$ As $G'''$ is a $1-$uniform $K-$dual of $F$
we have $\sum\limits_{i\neq j} \langle f_i, g'''_j \rangle \langle f_j, g'''_i \rangle = tr(K) - \dfrac{\left(tr(K)\right)^2}{N}. $	Thus, $\min\limits_{i \neq j} \langle f_i, g'''_j \rangle \langle f_j, g'''_i \rangle \leq \dfrac{tr(K) - \frac{\left(tr(K)\right)^2}{N}}{N(N-1)}< 0.$ Therefore, there exists $i_2 \neq j_2$ such that $\langle f_{i_2}, g'''_{j_2} \rangle \langle f_{j_2}, g'''_{i_2} \rangle \leq \dfrac{tr(K) - \frac{\left(tr(K)\right)^2}{N}}{N(N-1)}< 0.$ This in turn leads to  
$$r_{2} (F,G''')  \geq \left|\frac{tr(K)}{N} + \sqrt{ \langle g'''_{i_2}, f_{j_2} \rangle \langle g'''_{j_2}, f_{i_2} \rangle}   \right| = \sqrt{\left( \frac{tr(K)}{N}\right)^2 - \langle g'''_{i_2}, f_{j_2} \rangle \langle g'''_{j_2}, f_{i_2} \rangle} \;\geq \sqrt{\left(\dfrac{tr(K)}{N} \right)^2 - \dfrac{tr(K) - \frac{\left(tr(K)\right)^2}{N}}{N(N-1)}} . $$
		Hence, $\tilde{r}_{2}(F) = r_{2}(F,G) = \sqrt{\left(\dfrac{tr(K)}{N} \right)^2 - \dfrac{tr(K) - \frac{\left(tr(K)\right)^2}{N}}{N(N-1)}}$
	\end{proof}
	
	\noindent The following remark illustrates the effect of a unitary operator on spectral optimality. The proof approach for the above remark is similar to the one used in Theorem~\ref{thm4point3}.

	\begin{rem}\label{prop3point15}
		Let $F$ be a $K-$frame for $\mathcal{H}_n$  and  $U$ be a unitary operator on $\mathcal{H}_n$ such that $UK = KU.$ Then, $G \in \mathcal{R}_{i}(F)$ if and only if $UG \in \mathcal{R}_{i}(UF),\, i=1,2.$
	\end{rem}

	\noindent
\begin{example}
Consider the $K$-frame $F = \{f_1, f_2, f_3, f_4\}$ for $\mathbb{R}^3$, where
\[
f_1 = \begin{bmatrix} \sqrt{2} \\ 0 \\ 0 \end{bmatrix}, \quad
f_2 = \begin{bmatrix} \sqrt{2} \\ 0 \\ 0 \end{bmatrix}, \quad
f_3 = \begin{bmatrix} 0 \\ \frac{1}{\sqrt{2}} \\ \frac{1}{\sqrt{2}} \end{bmatrix}, \quad
f_4 = \begin{bmatrix} 0 \\ \frac{1}{\sqrt{2}} \\ -\frac{1}{\sqrt{2}} \end{bmatrix},
\]
and let
\[
K = \begin{bmatrix}
2 & 0 & 0 \\
0 & 1 & 0 \\
0 & 0 & 1
\end{bmatrix}.
\]
It is easy to verify that $F$ is a Parseval $K-$frame for $\mathbb{R}^3$. The pseudo inverse of $K$ is
\[
K^\dag = \begin{bmatrix}
\frac{1}{2} & 0 & 0 \\
0 & 1 & 0 \\
0 & 0 & 1
\end{bmatrix}.
\]

Hence, the canonical $K$-dual of $F$ is given by:
\[
K^\dag F = \left\{
\begin{bmatrix} \frac{1}{\sqrt{2}} \\ 0 \\ 0 \end{bmatrix},
\begin{bmatrix} \frac{1}{\sqrt{2}} \\ 0 \\ 0 \end{bmatrix},
\begin{bmatrix} 0 \\ \frac{1}{\sqrt{2}} \\ \frac{1}{\sqrt{2}} \end{bmatrix},
\begin{bmatrix} 0 \\ \frac{1}{\sqrt{2}} \\ -\frac{1}{\sqrt{2}} \end{bmatrix}
\right\}.
\]

We compute $\|f_i\| \cdot \|K^\dag f_i\| = 1$ for all $1 \leq i \leq 4$. Thus,
\[
\mathbb{O}_{1}(F, K^\dag F) = 1 = \frac{\operatorname{tr}(K)}{N},
\]
and by Proposition~\ref{thm4point1}, the pair $(F, K^\dag F)$ is a $1$-erasure optimal $(4,3)\;K-$dual pair with respect to the operator norm. Moreover, any $K-$dual $G = \{g_i\}_{i=1}^4$ of $F$ must be of the form:
\[
g_1 = \begin{bmatrix} \frac{1}{\sqrt{2}} + \alpha_1 \\ \beta_1 \\ \gamma_1 \end{bmatrix}, \quad
g_2 = \begin{bmatrix} \frac{1}{\sqrt{2}} - \alpha_1 \\ -\beta_1 \\ \gamma_1 \end{bmatrix}, \quad
g_3 = \begin{bmatrix} 0 \\ \frac{1}{\sqrt{2}} \\ \frac{1}{\sqrt{2}} \end{bmatrix}, \quad
g_4 = \begin{bmatrix} 0 \\ \frac{1}{\sqrt{2}} \\ -\frac{1}{\sqrt{2}} \end{bmatrix}, \quad
\alpha_1, \beta_1, \gamma_1 \in \mathbb{R}.
\]

If $G$ is a $1-$erasure optimal $K-$dual of $F$ under operator norm, then $\|f_i\| \cdot \|g_i\| = 1$ for all $i$. Now,
\begin{align*}
\|f_1\| \cdot \|g_1\| &= \sqrt{2} \cdot \sqrt{\left( \frac{1}{\sqrt{2}} + \alpha_1 \right)^2 + \beta_1^2 + \gamma_1^2}, \\
\|f_2\| \cdot \|g_2\| &= \sqrt{2} \cdot \sqrt{\left( \frac{1}{\sqrt{2}} - \alpha_1 \right)^2 + \beta_1^2 + \gamma_1^2}, \\
\|f_3\| \cdot \|g_3\| &= 1, \\
\|f_4\| \cdot \|g_4\| &= 1.
\end{align*}

Equating the norms implies $\alpha_1 = \beta_1 = \gamma_1 = 0$. Therefore, the canonical dual is the unique $1-$erasure optimal $K-$dual of $F$ under operator norm.

Additionally, we observe that $\langle f_i, K^\dag f_i \rangle = 1$ for all $i$, so that $r_1(F, K^\dag F) = 1$. By Theorem~\ref{thm5point2}, $(F, K^\dag F)$ is also a $1-$erasure spectrally optimal $K$-dual pair.

Now, suppose $G$ is another $1-$erasure spectrally optimal $K-$dual of $F$. Then,
\begin{align*}
\langle f_1, g_1 \rangle &= 1 + \sqrt{2}\alpha_1, \\
\langle f_2, g_2 \rangle &= 1 - \sqrt{2}\alpha_1, \\
\langle f_3, g_3 \rangle &= 1, \\
\langle f_4, g_4 \rangle &= 1.
\end{align*}
Thus, for $\langle f_i, g_i \rangle = 1$ to hold for all $i$, we must have $\alpha_1 = 0$. Hence, the family of duals
\[
\tilde{G} = \left\{
\begin{bmatrix} \frac{1}{\sqrt{2}} \\ \beta_1 \\ \gamma_1 \end{bmatrix},\;
\begin{bmatrix} \frac{1}{\sqrt{2}} \\ -\beta_1 \\ \gamma_1 \end{bmatrix},\;
\begin{bmatrix} 0 \\ \frac{1}{\sqrt{2}} \\ \frac{1}{\sqrt{2}} \end{bmatrix},\;
\begin{bmatrix} 0 \\ \frac{1}{\sqrt{2}} \\ -\frac{1}{\sqrt{2}} \end{bmatrix}
\right\}, \quad \beta_1, \gamma_1 \in \mathbb{R},
\]
forms a collection of $1-$erasure spectrally optimal $K$-duals of $F$. Therefore, the canonical dual $K^\dag F$ is not the unique $1-$erasure spectrally optimal $K-$dual of $F$.
\end{example}

\section*{Conclusion}

In this paper, we explored the problem of optimal reconstruction in the presence of erasures within the framework of $K$-frames. We addressed two main questions: identifying optimal $(N,n)$ $K$-dual pairs and determining optimal $K$-dual frames for a fixed Parseval $K$-frame under both operator norm and spectral radius criteria. We introduced and characterized $1$-uniform and $2$-uniform $K$-dual pairs, demonstrating their role in achieving optimality for one and two erasures, respectively. Tight lower bounds for reconstruction errors were established, and explicit conditions under which these bounds are attained were provided. For fixed Parseval $K$-frames, we derived both sufficient and necessary conditions under which the canonical dual is the unique optimal dual. Several examples illustrating the uniqueness and non-uniqueness of optimal duals were presented to highlight the geometric and algebraic structure of the solution space. Overall, our results extend and unify existing studies in classical frame theory and offer a robust operator theoretic framework for analyzing erasures and recovery in redundant systems. An interesting open direction is to study the case when $K$ is a negative definite operator, which remains unresolved within the current framework.

    \vspace{1cm}

	\noindent
	{\bf Acknowledgments:} 
	The authors are grateful to the Mohapatra Family Foundation and the College of Graduate Studies of the University of Central Florida for their support during this research. This work is also partially supported by NSF grant DMS-2105038.\\
    
    \noindent
	{\bf Data availability: } 
	The authors declare that no data has been used for this research.

\noindent
	{\bf Ethical approval: }  Not applicable.\\
\noindent
	{\bf Informed consent: } Not applicable.\\
\noindent
	{\bf Conflict of Interest:} Not Applicable.

\bibliographystyle{amsplain}

\begin{thebibliography}{10}

	
\bibitem{sali}
S.~Pehlivan, D.~Han and R.~Mohapatra,
\newblock {\em Linearly connected sequences and spectrally optimal dual frames for
  erasures}.
\newblock {J. Funct. Anal.} \textbf{265} (2013), 2855--2876.

\bibitem{ara}
Fahimeh Arabyani-Neyshaburi, Ali Akbar Arefijamaal, and Ghadir Sadeghi, Numerically and spectrally optimal
dual frames in hilbert spaces, Linear Algebra and its Applications 604 (2020), 52–71.

\bibitem{shan4}
S Arati, P Devaraj, and Shankhadeep Mondal, Optimal dual frames and dual pairs for probability modelled
erasures, Advances in Operator Theory 9 (2024), no. 2, 15.

\bibitem{li2}
Dongwei Li, Jinsong Leng, Tingzhu Huang, and Qing Gao, Frame expansions with probabilistic erasures, Digital
Signal Processing 72 (2018), 75–82.

\bibitem{gua}
Laura G˘ avrut¸a, Frames for operators, Applied and Computational Harmonic Analysis 32 (2012), 139–144.

\bibitem{du} Dandan Du and Yu-Can Zhu, Constructions of k-g-frames and tight k-g-frames in hilbert spaces, Bulletin of the
Malaysian Mathematical Sciences Society 43 (2020), no. 6, 4107–4122.

\bibitem{he}Miao He, Jinsong Leng, Dongwei Li, and Yuxiang Xu, Operator representations of k-frames: boundedness and
stability, J. Oper. Matrices 14 (2020), no. 4, 921–934.

\bibitem{li4} Ya-Nan Li and Yun-Zhang Li, Making and sharing k-dual frame pairs, Numerical Functional Analysis and
Optimization 42 (2021), no. 2, 155–179.

\bibitem{bar} Javad Baradaran, Jahangir Cheshmavar, Fereshte Nikahd, and Zahra Ghorbani, Some properties of c-k-g-frames in
hilbert c*-modules, Linear and Multilinear Algebra 71 (2023), no. 8, 1323–1337.

\bibitem{wang} Gang Wang, A study on the sum of k-frames in n-hilbert space, International Journal of Wavelets, Multiresolution
and Information Processing 22 (2024), no. 01, 2350035.

\bibitem{nami}Susan Nami, Gholamreza Rahimlou, Reza Ahmadi, and Mohammad Ali Jafarizadeh, Continuous k-frames and
their dual in hilbert spaces, (2020).


\bibitem{li3}Dongwei Li, Jinsong Leng, and Tingzhu Huang, Some properties of g-frames for hilbert space operators, Oper.
Matrices 11 (2017), no. 4, 1075–1085.


\bibitem{pill}K Raju Pillai and S Palaniammal, Certain investigations of k-frames in hilbert space and its application in
cryptography, Inter. Journal of Applied Engineering Research 11 (2016), no. 1, 31–35.

\bibitem{xiao} Xiangchun Xiao, Yucan Zhu, and Laura G˘avrut¸a, Some properties of k-frames in hilbert spaces, Results in
mathematics 63 (2013), no. 3-4, 1243–1255.


\bibitem{koba}Koji Kobayashi, Jun Kawahara, and Shuichi Miyazaki, Improved bounds for online k-frame throughput
maximization in network switches, IEICE Technical Report; IEICE Tech. Rep. 114 (2014), no. 19, 37–44.


\bibitem{duff}
R. J. Duffin and A. C. Schaeffer, A class of nonharmonic Fourier series, Transactions of the American Mathematical
Society 72 (1952), no. 2, 341–366.

\bibitem{shan}
S. Arati, P. Devaraj, and S. Mondal, Optimal dual frames and dual pairs for probability modelled erasures,
Advances in Operator Theory 9 (2024), 1–30.

\bibitem{bapat}
Ravindra B Bapat, Graphs and matrices, vol. 27, Springer, 2010.

\bibitem{bapat1}
RB Bapat, D Kalita, and S Pati, On weighted directed graphs, Linear Algebra and its Applications 436 (2012), no. 1, 99–111.


\bibitem{biggs}
Norman Biggs, Algebraic graph theory, no. 67, Cambridge university press, 1993.

\bibitem{peh}
S.~Pehlivan, D.~Han and R.~Mohapatra, \emph{Spectrally two-uniform
  frames for erasures}, Oper. Matrices \textbf{9}, no.~2, 383--399 (2015).


\bibitem{bodm1}
B.~G. Bodmann and V.~I. Paulsen,
\newblock {\em Frames, graphs and erasures}.
\newblock {Linear Algebra Appl.} \textbf{404} (2005), 118--146.

\bibitem{casa2}
P.~G. Casazza and J.~Kova{\v{c}}evi{\'c},
\newblock {\em Equal-norm tight frames with erasures}.
\newblock {Adv. Comput. Math.} \textbf{18} (2003), 387--430.


\bibitem{casa2}
P. G. Casazza and M. T. Leon, Existence and construction of finite frames with a given frame operator, Int. J. Pure Appl. Math. 63 (2010), no. 2, 149–157. MR 2683591


\bibitem{ole}
O.~Christensen,
\newblock {\em An Introduction to Frames and {Riesz} Bases}.
\newblock Birkh$\ddot{\text{a}}$user, 2016.

\bibitem{deep}
Deepshikha and A.~Samanta, \emph{Averaged numerically optimal dual frames for
  erasures}, Linear and Multilinear Algebra (2022), 1--16.


\bibitem{dev}
P.~Devaraj and S.~Mondal,
\newblock {\em Spectrally optimal dual frames for erasures}.
\newblock {Proc. Indian Acad. Sci. (Math. Sci.)} \textbf{133} (2023), https://doi.org/10.1007/s12044-023-00743-5 .

\bibitem{goya}
V.~K. Goyal, J.~Kova{\v{c}}evi{\'c} and J.~A. Kelner,
\newblock {\em Quantized frame expansions with erasures}.
\newblock {Appl. Comput. Harmon. Anal.} \textbf{10} (2001), 203--233.


 \bibitem{guo}
Krystal Guo and Bojan Mohar, Hermitian adjacency matrix of digraphs and mixed graphs, Journal of Graph Theory, 85 (2017), no. 1, 217–248.



\bibitem{holm}
R.~B. Holmes and V.~I. Paulsen,
\newblock {\em Optimal frames for erasures}.
\newblock {Linear Algebra Appl.} \textbf{377} (2004), 31--51.


\bibitem{jins}
J.~Lopez and D.~Han,
\newblock {\em Optimal dual frames for erasures}.
\newblock {Linear Algebra Appl.} 435 (2011), no. 6, 1464–1472.

\bibitem{leng}
J.~Leng, D.~Han and T.~Huang,
\newblock {\em Probability modelled optimal frames for erasures}.
\newblock {Linear Algebra Appl.} \textbf{438} (2013), 4222--4236.

\bibitem{li}
D. Li,  and J. Leng,  and T. Huang,  and Q. Gao,
\newblock {\em Frame expansions with probabilistic erasures}.
\newblock {Digital Signal Processing}  (2028), 75--82.

\bibitem{leng3}
J.~Leng, D.~Han and T.~Huang,
\newblock {\em Optimal dual frames for communication coding with probabilistic
  erasures}.
\newblock {IEEE Trans. Signal Process.} \textbf{59} (2011), 5380--5389.



\bibitem{li1}
D.~Li, J.~Leng and M.~He,
\newblock {\em Optimal dual frames for probabilistic erasures}.
\newblock {IEEE Access} \textbf{7} (2019), 2774--2781.


\bibitem{jerr}
J.~Lopez and D.~Han,
\newblock {\em Optimal dual frames for erasures}.
\newblock {Linear Algebra Appl.} \textbf{432} (2010), 471--482.



\bibitem{mehat}
 Ranjit Mehatari, M Rajesh Kannan, and Aniruddha Samanta, On the adjacency matrix of a complex unit gain
graph, Linear and Multilinear Algebra 70 (2022), no. 9, 1798–1813.


\bibitem{nisan}
 Noam Nisan and Avi Wigderson, On rank vs. communication complexity, Combinatorica 15 (1995), no. 4, 557–565

\bibitem{shi}
 Jianbo Shi and Jitendra Malik, Normalized cuts and image segmentation, IEEE Transactions on pattern analysis and machine intelligence 22 (2000), no. 8, 888–905


\bibitem{deep2}
  Deepshikha, Frames generated by graphs, https://arxiv.org/pdf/2405.16891 (2024).

\bibitem{deep1}
Deepshikha and Aniruddha Samanta, On spectrally optimal duals of frames generated by graphs,
https://arxiv.org/pdf/2406.00776 (2024)
\bibitem{zas}
Thomas Zaslavsky, Signed graphs, Discrete Applied Mathematics 4 (1982), no. 1, 47–74.

\bibitem{cass2} P. G. Casazza and M. T. Leon, Existence and construction of finite frames with a given frame operator, Int. J. Pure
Appl. Math. 63 (2010), no. 2, 149–157. MR 2683591	
	
\end{thebibliography}

\end{document}